\documentclass{article}
\usepackage{graphicx} 
\usepackage{amsmath}
\usepackage{amsfonts}
\usepackage{amssymb}
\usepackage[english]{babel}
\usepackage{amsthm}
\usepackage{tikz}
\usetikzlibrary{decorations.pathreplacing,calc,arrows.meta}
\usepackage{pgffor}
\usepackage{bbm}
\usepackage{thm-restate}
\usepackage[margin=1in]{geometry}

\newtheorem{proposition}{Proposition}
\newtheorem{definition}{Definition}
\newtheorem{theorem}{Theorem}
\newtheorem{lemma}{Lemma}
\newtheorem{corollary}{Corollary}
\newtheorem{remark}{Remark}

\title{\textbf{Slitherlink on Triangular Grids}}
\author{Charles Gong}
\date{}
\begin{document}

\maketitle

\begin{abstract}
Let $G$ be a plane graph and let $C$ be a cycle in $G$. For each finite face of $G$, count the number of edges of $C$ the face contains. We call this the Slitherlink signature of $C$. The symmetric difference $A$ of two cycles with the same signature is totally even, meaning every vertex is incident to an even number of edges in $A$ and every face contains an even number of edges in $A$. In this paper, we completely characterize totally even subsets in the triangular grid, and count the number of edges in any totally even subset of the triangular grid. We also show that the size of the symmetric difference of two cycles with the same signature in the triangular grid is divisible by $12$; this is best possible since 12 is the greatest common divisor of all the sizes of the symmetric difference between two cycles with the same signature in a triangular grid. 
\end{abstract}

\begin{center}
\section{Introduction}
\end{center}

\begin{definition}
Let $T_n$ denote the equilateral triangular grid, viewed as a graph, with three sides each containing $n$ edges. So $T_n$ is a plane graph that contains $\binom{n+2}{2}$ vertices and $3\binom{n+1}{2}$ edges.
\end{definition}

As an example, $T_3$ is given below.  \\

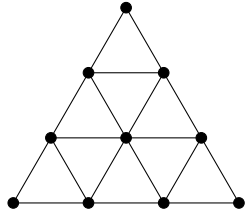
\begin{center}
\begin{tikzpicture}

\foreach \y in {0,1, ..., 2}
\foreach \x in {0, ..., \numexpr 2-\y} 
{
\draw ({\x + (\y / 2)}, {sqrt(3) * \y / 2}) -- ({\x + (\y / 2) + 1}, {sqrt(3) * \y / 2}) ;
\draw ({\x+(\y/2)},{sqrt(3) * \y / 2}) -- ({\x+(\y/2)+(1/2)},{(\y +1) * sqrt(3) / 2 }) ;
\draw ({\x + (\y / 2) + 1}, {sqrt(3) * \y / 2}) -- ({\x+(\y/2)+(1/2)},{(\y +1) * sqrt(3) / 2 }) ;
}

\foreach \y in {0,1,2,3} 
\foreach \x in {0,...,\numexpr 3 - \y}
{
\filldraw[black] ({\x + (\y / 2)},{sqrt(3) * \y / 2}) circle (2pt) ;
}

\end{tikzpicture}
\end{center}

In this paper, $G = (V,E)$ will always denote a finite and connected plane graph where every edge is a part of a cycle and lies in between two distinct faces (we consider the exterior to be a face, which we call the ``outside" or ``infinite" face). 

\begin{definition}[Beluhov \cite{bel}]
Consider a plane graph $G$ and a subset of the edges of $G$ that form a cycle, and label each \textit{finite} face (we ignore the infinite face/outside face) with the number of edges of that face appearing in the cycle. This we call the \textbf{Slitherlink signature} of the cycle. 
\end{definition}

\begin{definition}[Beluhov \cite{bel}]
Given a Slitherlink signature $S$, we define the \textbf{multiplicity} of $S$ to be the number of cycles with Slitherlink signature $S$.
\end{definition} 

For example, below is a Slitherlink signature on $T_5$ with multiplicity at least $2$. \\

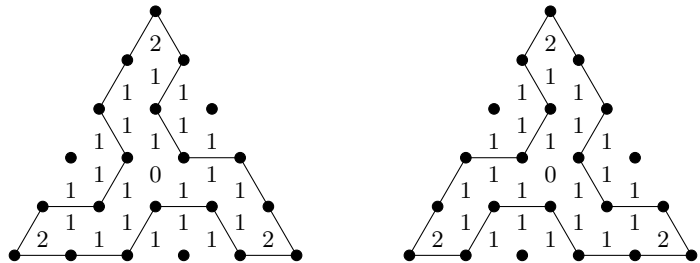
\begin{center}
\begin{tikzpicture}[declare function={a=0.75;}]
\foreach \y in {0,...,5} 
\foreach \x in {0,...,\numexpr 5 - \y}
{
\draw ({(\x + (\y / 2))*a},{(sqrt(3) * \y / 2)*a}) coordinate (\x\y) ;
\filldraw[black] ({(\x + (\y / 2))*a},{(sqrt(3) * \y / 2)*a}) circle (2pt) ;
}

\draw (00) -- (20) ;
\draw (20) -- (21) ;
\draw (21) -- (31) ;
\draw (31) -- (40) ;
\draw (40) -- (50) ;
\draw (32) -- (50) ;
\draw (32) -- (22) ;
\draw (22) -- (13) ;
\draw (13) -- (14) ;
\draw (14) -- (05) ;
\draw (05) -- (03) ;
\draw (03) -- (12) ;
\draw (12) -- (11) ;
\draw (11) -- (01) ;
\draw (01) -- (00) ;

\draw (barycentric cs:00=1,01=1,10=1) node{\small 2} ;
\draw (barycentric cs:01=1,11=1,10=1) node{\small 1} ;
\draw (barycentric cs:20=1,11=1,10=1) node{\small 1} ;
\draw (barycentric cs:20=1,11=1,21=1) node{\small 1} ;
\draw (barycentric cs:20=1,30=1,21=1) node{\small 1} ;
\draw (barycentric cs:31=1,30=1,21=1) node{\small 1} ;
\draw (barycentric cs:31=1,30=1,40=1) node{\small 1} ;
\draw (barycentric cs:31=1,41=1,40=1) node{\small 1} ;
\draw (barycentric cs:50=1,41=1,40=1) node{\small 2} ;
\draw (barycentric cs:01=1,02=1,11=1) node{\small 1} ;
\draw (barycentric cs:12=1,02=1,11=1) node{\small 1} ;
\draw (barycentric cs:12=1,21=1,11=1) node{\small 1} ;
\draw (barycentric cs:12=1,21=1,22=1) node{\small 0} ;
\draw (barycentric cs:31=1,21=1,22=1) node{\small 1} ;
\draw (barycentric cs:31=1,32=1,22=1) node{\small 1} ;
\draw (barycentric cs:31=1,32=1,41=1) node{\small 1} ;
\draw (barycentric cs:02=1,03=1,12=1) node{\small 1} ;
\draw (barycentric cs:13=1,03=1,12=1) node{\small 1} ;
\draw (barycentric cs:13=1,22=1,12=1) node{\small 1} ;
\draw (barycentric cs:13=1,22=1,23=1) node{\small 1} ;
\draw (barycentric cs:32=1,22=1,23=1) node{\small 1} ;
\draw (barycentric cs:03=1,04=1,13=1) node{\small 1} ;
\draw (barycentric cs:14=1,04=1,13=1) node{\small 1} ;
\draw (barycentric cs:14=1,23=1,13=1) node{\small 1} ;
\draw (barycentric cs:04=1,05=1,14=1) node{\small 2} ;

\foreach \y in {0,...,5} 
\foreach \x in {0,...,\numexpr 5 - \y}
{
\draw ({(\x + (\y / 2))*a+7*a},{(sqrt(3) * \y / 2)*a}) coordinate (\x\y a) ;
\filldraw[black] ({(\x + (\y / 2))*a + 7*a},{(sqrt(3) * \y / 2)*a}) circle (2pt) ;
}

\draw (00a) -- (02a) ;
\draw (02a) -- (12a) ;
\draw (12a) -- (13a) ;
\draw (13a) -- (04a) ;
\draw (04a) -- (05a) ;
\draw (05a) -- (23a) ;
\draw (23a) -- (22a) ;
\draw (22a) -- (31a) ;
\draw (31a) -- (41a) ;
\draw (41a) -- (50a) ;
\draw (50a) -- (30a) ;
\draw (30a) -- (21a) ;
\draw (21a) -- (11a) ;
\draw (11a) -- (10a) ;
\draw (10a) -- (00a) ;

\draw (barycentric cs:00a=1,01a=1,10a=1) node{\small 2} ;
\draw (barycentric cs:01a=1,11a=1,10a=1) node{\small 1} ;
\draw (barycentric cs:20a=1,11a=1,10a=1) node{\small 1} ;
\draw (barycentric cs:20a=1,11a=1,21a=1) node{\small 1} ;
\draw (barycentric cs:20a=1,30a=1,21a=1) node{\small 1} ;
\draw (barycentric cs:31a=1,30a=1,21a=1) node{\small 1} ;
\draw (barycentric cs:31a=1,30a=1,40a=1) node{\small 1} ;
\draw (barycentric cs:31a=1,41a=1,40a=1) node{\small 1} ;
\draw (barycentric cs:50a=1,41a=1,40a=1) node{\small 2} ;
\draw (barycentric cs:01a=1,02a=1,11a=1) node{\small 1} ;
\draw (barycentric cs:12a=1,02a=1,11a=1) node{\small 1} ;
\draw (barycentric cs:12a=1,21a=1,11a=1) node{\small 1} ;
\draw (barycentric cs:12a=1,21a=1,22a=1) node{\small 0} ;
\draw (barycentric cs:31a=1,21a=1,22a=1) node{\small 1} ;
\draw (barycentric cs:31a=1,32a=1,22a=1) node{\small 1} ;
\draw (barycentric cs:31a=1,32a=1,41a=1) node{\small 1} ;
\draw (barycentric cs:02a=1,03a=1,12a=1) node{\small 1} ;
\draw (barycentric cs:13a=1,03a=1,12a=1) node{\small 1} ;
\draw (barycentric cs:13a=1,22a=1,12a=1) node{\small 1} ;
\draw (barycentric cs:13a=1,22a=1,23a=1) node{\small 1} ;
\draw (barycentric cs:32a=1,22a=1,23a=1) node{\small 1} ;
\draw (barycentric cs:03a=1,04a=1,13a=1) node{\small 1} ;
\draw (barycentric cs:14a=1,04a=1,13a=1) node{\small 1} ;
\draw (barycentric cs:14a=1,23a=1,13a=1) node{\small 1} ;
\draw (barycentric cs:04a=1,05a=1,14a=1) node{\small 2} ;

\end{tikzpicture} \\

\bigskip 

Figure 1
\end{center}

Let $\mathbb{Z}^{+}$ denote the set of positive integers. We label the vertex in the $x$-th row from the bottom and $y$-th position from left to right by $(x,y) \in \mathbb{Z}^{+} \times \mathbb{Z}^{+}$. As an example, the labelling of the vertex set of $T_3$ is given below. \\

\begin{center}
\begin{tikzpicture}
\foreach \y in {0,1,2,3} 
\foreach \x in {0,...,\numexpr 3 - \y}
{
\filldraw[black] ({\x + (\y / 2)},{sqrt(3) * \y / 2}) circle (2pt) ;
\draw ({\x + (\y / 2)},{sqrt(3) * \y / 2}) node[yshift=-0.4cm]{\small (\pgfmathparse{\x + 1} \pgfmathprintnumber{\pgfmathresult}, \pgfmathparse{\y + 1} \pgfmathprintnumber{\pgfmathresult})} ;
}

\draw (1.5,-1.25) node{Figure 2} ;
\end{tikzpicture}
\end{center}

If you plot the points of $T_n$ with labels $(x,y) \in V(T_n) = \{(x,y) \in  \mathbb{Z}^{+} \times \mathbb{Z}^{+} : 1 \leq y \leq n, 1 \leq x \leq n+2-y \}$ in the Cartesian plane, you get an isosceles right triangle instead of an equilateral triangle. The reason why we chose an embedding of the triangular grid as an equilateral triangle is because some of our results are about symmetries that can only be seen when we embed the triangular grid as an equilateral triangle. It will also be useful for us to consider points outside of $T_n$ embedded in the plane. In particular, we care about points of the form $(x,0)$ for integer $2 \leq x \leq n$ and $(0,y)$ for integer $2 \leq y \leq n$ and $(n+3-y,y)$ for integer $2 \leq y \leq n$ in the coordinate system shown in Figure 2. See the picture of the vertex set of $T_3$ below, along with the points surrounding it. 

\begin{center}
\begin{tikzpicture}
\foreach \y in {0,1,2,3} 
\foreach \x in {0,...,\numexpr 3 - \y}
{
\filldraw[black] ({\x + (\y / 2)},{sqrt(3) * \y / 2}) circle (2pt) ;
\draw ({\x + (\y / 2)},{sqrt(3) * \y / 2}) node[yshift=-0.4cm]{\small (\pgfmathparse{\x + 1} \pgfmathprintnumber{\pgfmathresult}, \pgfmathparse{\y + 1} \pgfmathprintnumber{\pgfmathresult})} ;
}

\foreach \x in {1,2,3}
{
\draw ({\x + (-1 / 2)},{sqrt(3) * (-1) / 2}) circle (2pt) ;
\draw ({\x + (-1 / 2)},{sqrt(3) * (-1) / 2}) node[yshift=-0.4cm]{\small (\pgfmathparse{\x + 1} \pgfmathprintnumber{\pgfmathresult}, 0)} ;
}

\foreach \y in {1,2,3}
{
\draw ({-1 + (\y / 2)},{sqrt(3) * \y / 2}) circle (2pt) ;
\draw ({-1 + (\y / 2)},{sqrt(3) * \y / 2}) node[yshift=-0.4cm]{\small (\pgfmathparse{-1 + 1} \pgfmathprintnumber{\pgfmathresult}, \pgfmathparse{\y + 1} \pgfmathprintnumber{\pgfmathresult})} ;
}

\foreach \y in {1,2,3}
{
\draw ({4-\y + (\y / 2)},{sqrt(3) * \y / 2}) circle (2pt) ;
\draw ({4-\y  + (\y / 2)},{sqrt(3) * \y / 2}) node[yshift=-0.4cm]{\small (\pgfmathparse{4-\y  + 1} \pgfmathprintnumber{\pgfmathresult}, \pgfmathparse{\y + 1} \pgfmathprintnumber{\pgfmathresult})} ;
}
\draw (1.5,-2) node{Figure 3} ;
\end{tikzpicture}
\end{center}

\begin{definition} 
In $T_n$, we say the edges $\{(i,1),(i+1,1)\}$ for $i \in [n]$ and the vertices $(i,1)$ for $i \in [n+1]$ form the \textbf{bottom side} of $T_n$. Likewise, the edges $\{(1,j),(1,j+1)\}$ for $j \in [n]$ and the vertices $(1,j)$ for $j \in [n+1]$ form the \textbf{left side} of $T_n$, and the edges $\{(n+2-j,j),(n+1-j,j+1)\}$ for $j \in [n]$ and the vertices $(n+2-j,j)$ for $j \in [n+1]$ form the \textbf{right side} of $T_n$. 
\end{definition}

\begin{definition}[Beluhov \cite{bel}]
A \textbf{totally even} subset $A \subseteq E$ of a plane graph $G = (V,E)$ is a subset of edges such that for all $v \in V$, we have $v$ is incident to an even number of edges in $A$, and all faces in $G$ contain an even number of edges in $A$.
\end{definition}

 If $A,B \subseteq E$, we let $A \triangle B$ denote the symmetric difference of $A$ and $B$. We claim that given the edge sets $C_1, C_2 \subseteq E$ of two cycles with the same signature in $G$, we have $C_1 \triangle C_2$ is a totally even subset. Note if $|A|$ and $|B|$ are both even or both odd, then $|A \triangle B|$ is also even. Thus for every vertex $v \in V$, we know $v$ is incident to an even number of edges of $C_1$ and an even number of edges of $C_2$, namely $0$ or $2$ edges, and so $v$ is incident to an even number of edges in $C_1 \triangle C_2$. Furthermore for every face $f$ in $G$, since $C_1$ and $C_2$ have the same signature, $f$ contains the same number of edges in $C_1$ as in $C_2$, and therefore contains an even number of edges of $C_1 \triangle C_2$. Hence, $C_1 \triangle C_2$ is totally even. \\

In 2023 \cite{bel}, Beluhov studied Slitherlink signatures on $m$ by $n$ grids with square cells (call this a rectangular grid, and call it a square grid if $m = n$), and considered under what conditions on $m$ and $n$ is there a signature in the $m$ by $n$ grid with multiplicity greater than or equal to two. Hee also studied what multiplicities are possible for signatures on the $m$ by $n$ grid \cite{bel}. One result in particular he obtained was that an $m$ by $n$ grid with square cells has two distinct cycles with the same signature if and only if (1) $m, n \geq 3$, (2) $\gcd(m,n) \geq 2$, and (3) if $\gcd(m,n) = 2$ then either $4 \mid m$ or $4 \mid n$. Beluhov obtained various results leading up to this big theorem, and we have analogues of those results for triangular grids. Beluhov showed that totally even subsets of square grids had a four-fold symmetry (symmetric about the diagonals); we will show totally even subsets of triangular grids have a six-fold symmetry. Beluhov showed that the edge of a totally even subset on a side of a rectangular grid completely determines the totally even subset; we will show this for triangular grids. Beluhov completely characterized totally even subsets for rectangular grids; we will completely characterize totally even subsets of $T_n$. Beluhov counted the number of edges of any totally even subset of a rectangular grid; we will count the number of edges of any totally even subset of a triangular grid. Beluhov showed that the greatest common divisor of the sizes of the symmetric difference between two cycles with the same signature on a rectangular grid is 8; we will show this is $12$ for triangular grids. \\

There were many results Beluhov proved for rectangular grids that we wished to prove an analogue for triangular grids but couldn't. In particular, Beluhov obtained results about general constructions of pairs of totally even subsets in the $m$ by $n$ grid that worked for $m$ and $n$ satisfying (1) $m, n \geq 3$, (2) $\gcd(m,n) \geq 2$, and (3) if $\gcd(m,n) = 2$ then either $4 \mid m$ or $4 \mid n$. He also obtained a result that constructed pairs of distinct cycles with the same signature in the $am$ by $bn$ grid from a pair of distinct cycles with the same signature in the $m$ by $n$ grid. We were not able to prove or find any analogues of these results (we don't have any results about general constructions). 

\section{Complete Characterization of the Totally Even Subsets of $T_n$}
Let $G = (V,E)$ be a plane graph and $A \subseteq E$,  and let $\mathbbm{1}_A : E \rightarrow \mathbb{F}_2$ denote the indicator vector of $A$, that is, $\mathbbm{1}_A (e) = 1$ if $e \in A$ and $\mathbbm{1}_A (e) = 0$ if $e \not\in A$. Recall that if $|A|,|B|$ are both even, then $|A \triangle B|$ is even. Thus if $E_1,E_2 \subseteq E$ are both totally even, then $E_1 \triangle E_2$ is also totally even, since each vertex $v \in V$ is incident to an even number of edges in $E_1$ and an even number of edges in $E_2$, and each face $f$ in $G$ contains an even number of edges in $E_1$ and an even number of edges in $E_2$. In particular, this means that $\{ \mathbbm{1}_A \in \mathbb{F}_2^E : A \text{ is totally even}\}$ forms a subspace of $\mathbb{F}_2^E$, where $\mathbb{F}_2^E$ denotes the set of functions/vectors with domain $E$ and codomain $\mathbb{F}_2$. Henceforth, we will often equate a totally even subset with its indicator vector, so we might say ``these totally even subsets are linearly independent," but we are really referring to the indicator vectors of the totally even subsets in $\mathbb{F}_2^E$. \\

We first give three examples of totally even subsets of $T_6$. 

\begin{center}
\begin{tikzpicture}[declare function={a=0.5;}]
\foreach \y in {0,...,6} 
\foreach \x in {0,...,\numexpr 6 - \y}
{
\draw ({(\x + (\y / 2))*a},{(sqrt(3) * \y / 2)*a}) coordinate (\x\y) ;
\filldraw[black] ({(\x + (\y / 2))*a},{(sqrt(3) * \y / 2)*a}) circle (2pt) ;
}

\draw (00) -- (10) ;
\draw (10) -- (11) ;
\draw (11) -- (01) ;
\draw (01) -- (00) ;
\draw (02) -- (11) ;
\draw (02) -- (12) ;
\draw (12) -- (03) ;
\draw (03) -- (13) ;
\draw (13) -- (04) ;
\draw (04) -- (14) ;
\draw (14) -- (05) ;
\draw (05) -- (06) ;
\draw (06) -- (15) ;
\draw (15) -- (14) ;
\draw (14) -- (24) ;
\draw (24) -- (23) ;
\draw (23) -- (33) ;
\draw (33) -- (32) ;
\draw (32) -- (42) ;
\draw (42) -- (41) ;
\draw (41) -- (51) ;
\draw (51) -- (60) ;
\draw (60) -- (50) ;
\draw (50) -- (41) ;
\draw (41) -- (40) ;
\draw (40) -- (31) ;
\draw (31) -- (30) ;
\draw (30) -- (21) ;
\draw (21) -- (20) ;
\draw (20) -- (11) ;

\foreach \y in {0,...,6} 
\foreach \x in {0,...,\numexpr 6 - \y}
{
\draw ({(\x + (\y / 2))*a+8*a},{(sqrt(3) * \y / 2)*a}) coordinate (\x\y a) ;
\filldraw[black] ({(\x + (\y / 2))*a + 8*a},{(sqrt(3) * \y / 2)*a}) circle (2pt) ;
}

\draw (10a) -- (20a) ;
\draw (10a) -- (11a) ;
\draw (11a) -- (01a) ;
\draw (01a) -- (02a) ;
\draw (02a) -- (12a) ;
\draw (11a) -- (12a) ;
\draw (11a) -- (21a) ;
\draw (21a) -- (20a) ;
\draw (12a) -- (22a) ;
\draw (22a) -- (21a) ;
\draw (21a) -- (30a) ;
\draw (30a) -- (31a) ;
\draw (31a) -- (40a) ;
\draw (40a) -- (50a) ;
\draw (50a) -- (41a) ;
\draw (41a) -- (31a) ;
\draw (22a) -- (31a) ;
\draw (41a) -- (51a) ;
\draw (51a) -- (42a) ;
\draw (42a) -- (32a) ;
\draw (32a) -- (41a) ;
\draw (22a) -- (32a) ;
\draw (12a) -- (03a) ;
\draw (03a) -- (13a) ;
\draw (13a) -- (04a) ;
\draw (04a) -- (05a) ;
\draw (05a) -- (14a) ;
\draw (14a) -- (13a) ;
\draw (14a) -- (15a) ;
\draw (15a) -- (24a) ;
\draw (24a) -- (23a) ;
\draw (23a) -- (14a) ;
\draw (13a) -- (22a) ;
\draw (22a) -- (23a) ;
\draw(23a) -- (33a) ;
\draw (33a) -- (32a) ;

\foreach \y in {0,...,6} 
\foreach \x in {0,...,\numexpr 6 - \y}
{
\draw ({(\x + (\y / 2))*a+16*a},{(sqrt(3) * \y / 2)*a}) coordinate (\x\y b) ;
\filldraw[black] ({(\x + (\y / 2))*a + 16*a},{(sqrt(3) * \y / 2)*a}) circle (2pt) ;
}
\draw (20b) -- (30b) ;
\draw (30b) -- (40b) ;
\draw (40b) -- (31b) ;
\draw (31b) -- (41b) ;
\draw (41b) -- (32b) ;
\draw (32b) -- (42b) ;
\draw (42b) -- (33b) ;
\draw (33b) -- (24b) ;
\draw (24b) -- (23b) ;
\draw (23b) -- (14b) ;
\draw (14b) -- (13b) ;
\draw (13b) -- (04b) ;
\draw (04b) -- (03b) ;
\draw (03b) -- (02b) ;
\draw (02b) -- (12b) ;
\draw (12b) -- (11b) ;
\draw (11b) -- (21b) ;
\draw (21b) -- (20b) ;

\draw (11*a, -1) node{\normalsize Figure 4} ;

\end{tikzpicture}
\end{center}

We claim that these three totally even subsets form a basis for all totally even subsets of $T_6$, so that in $T_6$ there are $2^3 = 8$ distinct totally even subsets. In general, $T_n$ will have $2^{\lfloor \frac{n}{2} \rfloor}$ distinct totally even subsets, which we will prove in a series of theorems below. First note that in Figure 4, all three of the totally even subsets are symmetric about the line containing the vertices $(2,5)$ and $(1,7)$, and thus since they form a basis all totally even subsets of $T_6$ are symmetric symmetric about the line containing the vertices $(2,5)$ and $((1,7)$. In general, all totally even subsets of $T_n$ are symmetric about the line containing the vertices $(1,n+1)$ and $(2,n-1)$: we will refer to this line as \textbf{the middle}. In the case $n = 1$ the middle is the line containing the vertex $(1,2)$ and the midpoint of the two vertices $(1,1)$ and $(2,1)$. \\

\begin{definition}
In $T_n$, the \textbf{middle} is the line containing the vertex $(1,n+1)$ and the midpoint of the two vertices $(1,n)$ and $(2,n)$.
\end{definition}

\begin{theorem}
Any totally even subset $A$ of $T_n$ is symmetric about the middle, that is, 

$$\{(x_1,y_1),(x_2,y_2)\} \in A$$

\noindent
if and only if 

$$\{(n+3-y_1-x_1,y_1),(n+3-y_2-x_2,y_2)\} \in A.$$
\end{theorem}

\begin{proof}
Let $A$ be a totally even subset of the edges of $T_n$. We first show the top two edges $\{(1,n+1),(1,n))\}$ and $\{(1,n+1),(2,n)\}$ are symmetric about the middle and then work our way downwards. In the pictures below, filled in lines refer to edges we've established to be symmetric about the middle, and dashed lines refer to edges we're currently trying to establish to be symmetric about the middle. 

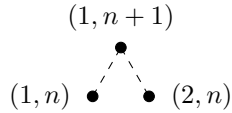
\begin{center}
\begin{tikzpicture}[declare function={a=0.75;}]
\foreach \y in {0,...,1} 
\foreach \x in {0,...,\numexpr 1 - \y}
{
\draw ({(\x + (\y / 2))*a},{(sqrt(3) * \y / 2)*a}) coordinate (\x\y) ;
\filldraw[black] ({(\x + (\y / 2))*a},{(sqrt(3) * \y / 2)*a}) circle (2pt) ;
}

\draw[dashed] (00) -- (01) ;
\draw[dashed] (01) -- (10) ;

\draw (01) node[yshift=0.4cm]{ $(1,n+1)$} ;
\draw (00) node[xshift=-0.7cm]{$(1,n)$} ;
\draw (10) node[xshift=0.7cm]{$(2,n)$} ;
\end{tikzpicture}
\end{center}

Edges $\{(1,n+1),(1,n))\}$ and $\{(1,n+1),(2,n)\}$ must either both be in $A$ or both not be in $A$. This is because if one is in $A$ and the other isn't, this would imply vertex $(1,n+1)$ is incident to $1$ edge of $A$, and $1$ is odd. Now we consider the edge $\{(1,n),(2,n)\}$. \\

\begin{center}
\begin{tikzpicture}[declare function={a=0.75;}]
\foreach \y in {0,...,1} 
\foreach \x in {0,...,\numexpr 1 - \y}
{
\draw ({(\x + (\y / 2))*a},{(sqrt(3) * \y / 2)*a}) coordinate (\x\y) ;
\filldraw[black] ({(\x + (\y / 2))*a},{(sqrt(3) * \y / 2)*a}) circle (2pt) ;
}

\draw (00) -- (01) ;
\draw (01) -- (10) ;
\draw[dashed] (00) -- (10) ;

\draw (00) node[xshift=-0.7cm]{$(1,n)$} ;
\draw (10) node[xshift=0.7cm]{$(2,n)$} ;
\end{tikzpicture}
\end{center}

\noindent
Since the edge $\{(1,n),(2,n)\}$ lies on the middle, it is symmetric about the middle (in fact, we will show in a corollary no edges that lie on the middle can be in $A$). We call the argument made in this paragraph \textbf{argument 1}. \\

Now we consider the two edges $\{(1,n),(2,n-1)\}$ and $\{(2,n),(2,n-1)\}$. \\

\begin{center}
\begin{tikzpicture}[declare function={a=0.75;}]
\foreach \y in {0,...,2} 
\foreach \x in {0,...,\numexpr 2 - \y}
{
\draw ({(\x + (\y / 2))*a},{(sqrt(3) * \y / 2)*a}) coordinate (\x\y) ;
\filldraw[black] ({(\x + (\y / 2))*a},{(sqrt(3) * \y / 2)*a}) circle (2pt) ;
}

\draw (01) -- (02) ;
\draw (02) -- (11) ;
\draw (01) -- (11) ;
\draw[dashed] (01) -- (10) ;
\draw[dashed] (11) -- (10) ;

\draw (01) node[xshift=-0.7cm]{$(1,n)$} ;
\draw (11) node[xshift=0.7cm]{$(2,n)$} ;
\draw (10) node[yshift=-0.4cm]{$(2,n-1)$} ;

\end{tikzpicture}
\end{center}

\noindent
Without loss of generality, assume for the sake of contradiction $\{(1,n),(2,n-1)\} \in A$ and $\{(2,n),(2,n-1)\} \not \in A$. Since all faces in $T_n$ either contain $2$ edges in $A$ or $0$ edges in $A$, and since we know the face with vertices $(1,n),(2,n),(2,n-1)$ has at least $1$ edge of $A$, it must have $2$ edges of $A$. Thus, $\{(1,n),(2,n)\} \in A$. By the same reasoning, we now know the face with vertices $(1,n+1),(1,n),(2,n)$ has two edges in $A$, and therefore exactly one of $\{(1,n+1),(1,n))\}, \{(1,n+1),(2,n)\}$ is in $A$, contradicting what we established before. So the two edges $\{(1,n),(2,n-1)\}$ and $\{(2,n),(2,n-1)\}$ are symmetric about the middle. We call the argument made in this paragraph \textbf{argument 2}. \\

We now consider the edges $\{(1,n),(1,n-1)\}$ and $\{(2,n),(3,n-1)\}$. \\

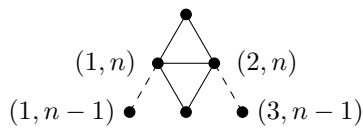
\begin{center}
\begin{tikzpicture}[declare function={a=0.75;}]
\foreach \y in {0,...,2} 
\foreach \x in {0,...,\numexpr 2 - \y}
{
\draw ({(\x + (\y / 2))*a},{(sqrt(3) * \y / 2)*a}) coordinate (\x\y) ;
\filldraw[black] ({(\x + (\y / 2))*a},{(sqrt(3) * \y / 2)*a}) circle (2pt) ;
}

\draw (01) -- (02) ;
\draw (02) -- (11) ;
\draw (01) -- (11) ;
\draw (01) -- (10) ;
\draw (11) -- (10) ;
\draw[dashed] (00) -- (01) ;
\draw[dashed] (20) -- (11) ;

\draw (01) node[xshift=-0.7cm]{$(1,n)$} ;
\draw (11) node[xshift=0.7cm]{$(2,n)$} ;
\draw (00) node[xshift=-0.9cm]{$(1,n-1)$} ;
\draw (20) node[xshift=0.9cm]{$(3,n-1)$} ;

\end{tikzpicture}
\end{center}

\noindent
Note if we know whether the three edges $\{(1,n),(1,n+1)\}$, $\{(1,n),(2,n)\}$, $\{(1,n),(2,n-1)\}$ are in $A$, then we know whether $\{(1,n),(1,n-1)\}$ is in $A$ because vertex $(1,n)$ must be incident to an even number of edges in $A$. Likewise, if we know whether the three edges $\{(2,n),(1,n+1)\}$, $\{(2,n),(1,n)\}$, $\{(2,n),(2,n-1)\}$ are in $A$, then we know whether $\{(2,n),(3,n-1)\}$ is in $A$ because vertex $(2,n)$ must be incident to an even number of edges in $A$. In general, if we know whether all of the edges incident to a vertex $v$ are in $A$ except one of the edges $e$, then can also determine whether $e \in A$. By the previously established symmetries, this means $\{(1,n),(1,n-1)\}$ and $\{(2,n),(3,n-1)\}$ are symmetric about the middle. We call the argument made in this paragraph \textbf{argument 3}. \\

Now we consider the two edges $\{(1,n-1),(2,n-1)\}$ and $\{(2,n-1),(3,n-1)$. \\

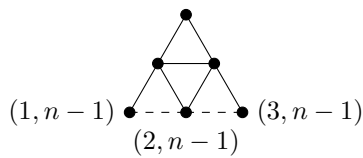
\begin{center}
\begin{tikzpicture}[declare function={a=0.75;}]
\foreach \y in {0,...,2} 
\foreach \x in {0,...,\numexpr 2 - \y}
{
\draw ({(\x + (\y / 2))*a},{(sqrt(3) * \y / 2)*a}) coordinate (\x\y) ;
\filldraw[black] ({(\x + (\y / 2))*a},{(sqrt(3) * \y / 2)*a}) circle (2pt) ;
}

\draw (01) -- (02) ;
\draw (02) -- (11) ;
\draw (01) -- (11) ;
\draw (01) -- (10) ;
\draw (11) -- (10) ;
\draw (00) -- (01) ;
\draw (20) -- (11) ;
\draw[dashed] (00) -- (10) ;
\draw[dashed] (10) -- (20) ;

\draw (00) node[xshift=-0.9cm]{$(1,n-1)$} ;
\draw (20) node[xshift=0.9cm]{$(3,n-1)$} ;
\draw (10) node[yshift=-0.4cm]{$(2,n-1)$} ;

\end{tikzpicture}
\end{center}

\noindent
Note if we know whether the edges $\{(1,n-1),(1,n)\}$ and $\{(1,n),(2,n-1)\}$ are in $A$, then we know whether $\{(1,n-1),(2,n-1)\}$ is in $A$ because the face with vertices $(1,n-1),(1,n),(2,n-1)$ must have an even number of edges in $A$. Likewise, if we know whether the edges $\{(2,n-1),(2,n)\}$ and $\{(2,n),(3,n-1)\}$ are in $A$, then we know whether $\{(2,n-1),(3,n-1)\}$ is in $A$ because the face with vertices $(2,n-1),(2,n),(3,n-1)$ must have an even number of edges in $A$. In general, if we know whether all the edges of a face is in $A$ except for one of the edges $e$, then we can also determine whether $e \in A$. By the previously established symmetries, this means $\{(1,n-1),(2,n-1)\}$ and $\{(2,n-1),(3,n-1)\}$ are symmetric about the middle. We call the argument made in this paragraph \textbf{argument 4}. \\

Now we consider the edges $\{(2,n-1),(2,n-2)\}$ and $\{(2,n-1),(3,n-2)\}$. 

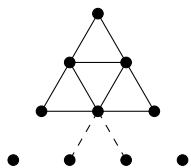
\begin{center}
\begin{tikzpicture}[declare function={a=0.75;}]
\foreach \y in {0,...,3} 
\foreach \x in {0,...,\numexpr 3 - \y}
{
\draw ({(\x + (\y / 2))*a},{(sqrt(3) * \y / 2)*a}) coordinate (\x\y) ;
\filldraw[black] ({(\x + (\y / 2))*a},{(sqrt(3) * \y / 2)*a}) circle (2pt) ;
}

\draw (02) -- (03) ;
\draw (03) -- (12) ;
\draw (02) -- (12) ;
\draw (02) -- (11) ;
\draw (12) -- (11) ;
\draw (01) -- (02) ;
\draw (21) -- (12) ;
\draw (01) -- (11) ;
\draw (11) -- (21) ;
\draw[dashed] (11) -- (10) ;
\draw[dashed] (11) -- (20) ;

\end{tikzpicture}
\end{center}

\noindent
By the symmetry established on the pairs of edges $\{(1,n),(2,n-1)\},\{(2,n),(2,n-1)\}$ and \linebreak $\{(1,n-1),(2,n-1)\},\{(2,n-1),(3,n-1)\}$, we know the edges $\{(2,n-1),(2,n-2)\}$ and $\{(2,n-1),(3,n-2)\}$ must either both be in $A$ or both not be in $A$, because vertex $(2,n-1)$ must have an even number of edges in $A$. We call the argument made in this paragraph \textbf{argument 5}. \\

We now consider the edge $\{(2,n-2),(3,n-2)\}$. 

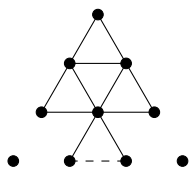
\begin{center}
\begin{tikzpicture}[declare function={a=0.75;}]
\foreach \y in {0,...,3} 
\foreach \x in {0,...,\numexpr 3 - \y}
{
\draw ({(\x + (\y / 2))*a},{(sqrt(3) * \y / 2)*a}) coordinate (\x\y) ;
\filldraw[black] ({(\x + (\y / 2))*a},{(sqrt(3) * \y / 2)*a}) circle (2pt) ;
}

\draw (02) -- (03) ;
\draw (03) -- (12) ;
\draw (02) -- (12) ;
\draw (02) -- (11) ;
\draw (12) -- (11) ;
\draw (01) -- (02) ;
\draw (21) -- (12) ;
\draw (01) -- (11) ;
\draw (11) -- (21) ;
\draw (11) -- (10) ;
\draw (11) -- (20) ;
\draw[dashed] (10) -- (20) ;

\end{tikzpicture}
\end{center}

\noindent
The edge $\{(2,n-2),(3,n-2)\}$ is symmetric about the middle by an argument analogous to argument 1. \\

We now consider the edges $\{(1,n-1),(2,n-2)\}$ and $\{(3,n-1),(3,n-2)\}$, \\

\begin{center}
\begin{tikzpicture}[declare function={a=0.75;}]
\foreach \y in {0,...,3} 
\foreach \x in {0,...,\numexpr 3 - \y}
{
\draw ({(\x + (\y / 2))*a},{(sqrt(3) * \y / 2)*a}) coordinate (\x\y) ;
\filldraw[black] ({(\x + (\y / 2))*a},{(sqrt(3) * \y / 2)*a}) circle (2pt) ;
}

\draw (02) -- (03) ;
\draw (03) -- (12) ;
\draw (02) -- (12) ;
\draw (02) -- (11) ;
\draw (12) -- (11) ;
\draw (01) -- (02) ;
\draw (21) -- (12) ;
\draw (01) -- (11) ;
\draw (11) -- (21) ;
\draw (11) -- (10) ;
\draw (11) -- (20) ;
\draw (10) -- (20) ;
\draw[dashed] (01) -- (10) ;
\draw[dashed] (21) -- (20) ;

\end{tikzpicture}
\end{center}

\noindent
The edges $\{(1,n-1),(2,n-2)\}$ and $\{(3,n-1),(3,n-2)\}$ are symmetric about the middle by an argument analogous to argument 4. \\

We now consider the edges $\{(1,n-1),(1,n-2)\}$ and $\{(3,n-1),(4,n-2)\}$. \\

\begin{center}
\begin{tikzpicture}[declare function={a=0.75;}]
\foreach \y in {0,...,3} 
\foreach \x in {0,...,\numexpr 3 - \y}
{
\draw ({(\x + (\y / 2))*a},{(sqrt(3) * \y / 2)*a}) coordinate (\x\y) ;
\filldraw[black] ({(\x + (\y / 2))*a},{(sqrt(3) * \y / 2)*a}) circle (2pt) ;
}

\draw (02) -- (03) ;
\draw (03) -- (12) ;
\draw (02) -- (12) ;
\draw (02) -- (11) ;
\draw (12) -- (11) ;
\draw (01) -- (02) ;
\draw (21) -- (12) ;
\draw (01) -- (11) ;
\draw (11) -- (21) ;
\draw (11) -- (10) ;
\draw (11) -- (20) ;
\draw (10) -- (20) ;
\draw (01) -- (10) ;
\draw (21) -- (20) ;
\draw[dashed] (00) -- (01) ;
\draw[dashed] (30) -- (21) ;

\end{tikzpicture}
\end{center}

\noindent
The edges $\{(1,n-1),(1,n-2)\}$ and $\{(3,n-1),(4,n-2)\}$ are symmetric about the middle by an argument analogous to argument 3. \\

We now consider the edges $\{(1,n-2),(2,n-2)\}$ and $\{(3,n-2),(4,n-2)\}$. \\

\begin{center}
\begin{tikzpicture}[declare function={a=0.75;}]
\foreach \y in {0,...,3} 
\foreach \x in {0,...,\numexpr 3 - \y}
{
\draw ({(\x + (\y / 2))*a},{(sqrt(3) * \y / 2)*a}) coordinate (\x\y) ;
\filldraw[black] ({(\x + (\y / 2))*a},{(sqrt(3) * \y / 2)*a}) circle (2pt) ;
}

\draw (02) -- (03) ;
\draw (03) -- (12) ;
\draw (02) -- (12) ;
\draw (02) -- (11) ;
\draw (12) -- (11) ;
\draw (01) -- (02) ;
\draw (21) -- (12) ;
\draw (01) -- (11) ;
\draw (11) -- (21) ;
\draw (11) -- (10) ;
\draw (11) -- (20) ;
\draw (10) -- (20) ;
\draw (01) -- (10) ;
\draw (21) -- (20) ;
\draw (00) -- (01) ;
\draw (30) -- (21) ;
\draw[dashed] (00) -- (10) ;
\draw[dashed] (20) -- (30) ;

\end{tikzpicture}
\end{center}

\noindent
The edges $\{(1,n-2),(2,n-2)\}$ and $\{(3,n-2),(4,n-2)\}$ are symmetric about the middle by an argument analogous to argument 4. \\

We now consider the edges $\{(2,n-2),(3,n-3)\}$ and $\{(3,n-2),(3,n-3)\}$. \\

\begin{center}
\begin{tikzpicture}[declare function={a=0.75;}]
\foreach \y in {0,...,4} 
\foreach \x in {0,...,\numexpr 4 - \y}
{
\draw ({(\x + (\y / 2))*a},{(sqrt(3) * \y / 2)*a}) coordinate (\x\y) ;
\filldraw[black] ({(\x + (\y / 2))*a},{(sqrt(3) * \y / 2)*a}) circle (2pt) ;
}

\draw (03) -- (04) ;
\draw (04) -- (13) ;
\draw (03) -- (13) ;
\draw (03) -- (12) ;
\draw (13) -- (12) ;
\draw (02) -- (03) ;
\draw (22) -- (13) ;
\draw (02) -- (12) ;
\draw (12) -- (22) ;
\draw (12) -- (11) ;
\draw (12) -- (21) ;
\draw (11) -- (21) ;
\draw (02) -- (11) ;
\draw (22) -- (21) ;
\draw (01) -- (02) ;
\draw (31) -- (22) ;
\draw (01) -- (11) ;
\draw (21) -- (31) ;
\draw[dashed] (20) -- (11) ;
\draw[dashed] (20) -- (21) ;

\end{tikzpicture}
\end{center}

\noindent
The edges $\{(2,n-2),(3,n-3)\}$ and $\{(3,n-2),(3,n-3)\}$ are symmetric about the middle by an argument analogous to argument 2. One can continue downwards the triangle and see $A$ is symmetric about the middle using arguments analogous to arguments 1, 2, 3, 4 and 5. At each layer, start from the middle and work your way outwards.
\end{proof}

\begin{corollary} 
Any edge lying on the middle in $T_n$ is not in any totally even subset $A$. 
\end{corollary}

\begin{proof}
Assume otherwise for the sake of contradiction, and let the edge that lies on the middle be \linebreak $\{(x,y),(x+1,y)\}$. This would force exactly one of the edges $\{(x,y),(x,y+1)\}, \{(x+1,y),(x,y+1)\}$ to be in $A$, contradicting the symmetry about the middle established in Theorem 1.
\end{proof}

Note that there are analogous proofs for the fact that any totally even subset of $T_n$ is symmetric about the line containing the vertices $(1,1)$ and $(2,2)$ and about the line containing the vertices $(n+1,1)$ and $(n-1,2)$. These three axes of symmetry also imply rotational symmetry: just flip $T_n$ about any two of the three axes of symmetry. See the picture of $T_n$ below.

\begin{center}
\begin{tikzpicture}[declare function={a=4;}]
\filldraw[black] (0,0) circle (2pt) ;
\filldraw[black] (a,0) circle (2pt) ;
\filldraw[black] (a/2, {sqrt(3)*a/2}) circle (2pt) ;

\draw (0,0) node[yshift=-0.6cm]{$(1,1)$} ;
\draw (a,0) node[yshift=-0.6cm]{$(n+1,1)$} ;
\draw (a/2, {sqrt(3)*a/2}) node[yshift=0.6cm]{$(1,n+1)$} ;

\draw[dashed] (a/2,-0.4) -- (a/2,{(sqrt(3)*a/2)+0.4}) ;
\draw[dashed] ({-sqrt(0.12)},{-sqrt(0.04)}) -- ({(3*a/4)+sqrt(0.12)},{(sqrt(3)*a/4)+sqrt(0.04)}) ;
\draw[dashed] ({(a/4) - sqrt(0.12)},{(sqrt(3)*a/4)+sqrt(0.04)}) -- ({a+sqrt(0.12)},{-sqrt(0.04)}) ;

\draw (0,0) -- (a,0) ;
\draw (a,0) -- (a/2,{sqrt(3)*a/2}) ;
\draw (a/2,{sqrt(3)*a/2}) -- (0,0) ;

\draw (a/2,-2) node{Figure 5} ;
\end{tikzpicture}
\end{center}

Note how in Figure 5 the three dashed lines (axes of symmetry) split the equilateral triangle into six triangular regions with the same side lengths. Given a totally even subset $A$ of $T_n$, the edges embedded in these regions are symmetric to each other, that is, there exists an isometric map from any of the six regions to another that preserves the edges. 
\begin{theorem}
Let $A$ and $B$ be totally even subsets of $T_n$. If $A$ and $B$ agree on the bottom side, then $A = B$. In other words, if $\mathbbm{1}_A(e) = \mathbbm{1}_B(e)$ for all $e$ of the form $e = \{(i,1),(i+1,1)\}$ where $i \in [n]$, then $\mathbbm{1}_A = \mathbbm{1}_B$.
\end{theorem}

\begin{proof}
We will show that knowing that $A$ is totally even and the values of $\mathbbm{1}_A(e)$ for all $e$ on the bottom side, the indicator vector $\mathbbm{1}_A$ is uniquely determined. In the general case for any totally even subset of $T_n$, we start at the bottom layer and move upwards, and at each layer we scan from left to right checking whether an edge is in $A$ or not. The key point is that at any step in the process (until we've determined all edges) there is a face or a vertex with exactly one edge $e$ remaining that we don't know to be in $A$ or not, however the assumption that $A$ is totally even forces $e$ to be in $A$ in the case the vertex or face has an odd number of edges (excluding $e$) in $A$, and forces this last edge to not be in $A$ otherwise. \\

We show it for an example in $T_4$, and then discuss the general case a bit more precisely at the end. Solid lines are edges we know to be in $A$, blank lines are edges we know not to be in $A$, and dashed lines are edges we are still considering. We work our way from the bottom and upwards, and left to right. We first consider the edge $\{(1,1),(1,2)\}$. \\

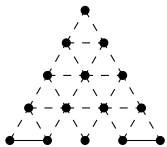
\begin{center}
\begin{tikzpicture}[declare function={a=0.5;}]
\foreach \y in {0,...,4} 
\foreach \x in {0,...,\numexpr 4 - \y}
{
\draw ({(\x + (\y / 2))*a},{(sqrt(3) * \y / 2)*a}) coordinate (\x\y) ;
\filldraw[black] ({(\x + (\y / 2))*a},{(sqrt(3) * \y / 2)*a}) circle (1.5pt) ;
}

\draw[dashed] (03) -- (04) ;
\draw[dashed] (04) -- (13) ;
\draw[dashed] (03) -- (13) ;
\draw[dashed] (03) -- (12) ;
\draw[dashed] (13) -- (12) ;
\draw[dashed] (02) -- (03) ;
\draw[dashed] (22) -- (13) ;
\draw[dashed] (02) -- (12) ;
\draw[dashed] (12) -- (22) ;
\draw[dashed] (12) -- (11) ;
\draw[dashed] (12) -- (21) ;
\draw[dashed] (11) -- (21) ;
\draw[dashed] (02) -- (11) ;
\draw[dashed] (22) -- (21) ;
\draw[dashed] (01) -- (02) ;
\draw[dashed] (31) -- (22) ;
\draw[dashed] (01) -- (11) ;
\draw[dashed] (21) -- (31) ;
\draw[dashed] (20) -- (11) ;
\draw[dashed] (20) -- (21) ;
\draw[dashed] (00) -- (01) ;
\draw[dashed] (01) -- (10) ;
\draw[dashed] (10) -- (11) ;
\draw (00) -- (10) ;
\draw (30) -- (40) ;
\draw[dashed] (21)-- (30) ;
\draw[dashed] (30) -- (31) ;
\draw[dashed] (31) -- (40) ;

\end{tikzpicture}
\end{center}

We must have $\{(1,1),(1,2)\} \in A$, otherwise vertex $(1,1)$ will be incident to an odd number of edges in $A$. Next we consider the edge $\{(2,1),(1,2)\}$. \\

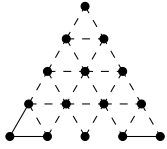
\begin{center}
\begin{tikzpicture}[declare function={a=0.5;}]
\foreach \y in {0,...,4} 
\foreach \x in {0,...,\numexpr 4 - \y}
{
\draw ({(\x + (\y / 2))*a},{(sqrt(3) * \y / 2)*a}) coordinate (\x\y) ;
\filldraw[black] ({(\x + (\y / 2))*a},{(sqrt(3) * \y / 2)*a}) circle (1.5pt) ;
}

\draw[dashed] (03) -- (04) ;
\draw[dashed] (04) -- (13) ;
\draw[dashed] (03) -- (13) ;
\draw[dashed] (03) -- (12) ;
\draw[dashed] (13) -- (12) ;
\draw[dashed] (02) -- (03) ;
\draw[dashed] (22) -- (13) ;
\draw[dashed] (02) -- (12) ;
\draw[dashed] (12) -- (22) ;
\draw[dashed] (12) -- (11) ;
\draw[dashed] (12) -- (21) ;
\draw[dashed] (11) -- (21) ;
\draw[dashed] (02) -- (11) ;
\draw[dashed] (22) -- (21) ;
\draw[dashed] (01) -- (02) ;
\draw[dashed] (31) -- (22) ;
\draw[dashed] (01) -- (11) ;
\draw[dashed] (21) -- (31) ;
\draw[dashed] (20) -- (11) ;
\draw[dashed] (20) -- (21) ;
\draw (00) -- (01) ;
\draw[dashed] (01) -- (10) ;
\draw[dashed] (10) -- (11) ;
\draw (00) -- (10) ;
\draw (30) -- (40) ;
\draw[dashed] (21)-- (30) ;
\draw[dashed] (30) -- (31) ;
\draw[dashed] (31) -- (40) ;

\end{tikzpicture}
\end{center}

We must have $\{(2,1),(1,2)\} \not\in A$, otherwise the face with vertices $(1,1),(2,1),(1,2)$ will have three edges in $A$. Next we consider the edge $\{(2,1),(2,2)\}$. \\

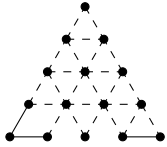
\begin{center}
\begin{tikzpicture}[declare function={a=0.5;}]
\foreach \y in {0,...,4} 
\foreach \x in {0,...,\numexpr 4 - \y}
{
\draw ({(\x + (\y / 2))*a},{(sqrt(3) * \y / 2)*a}) coordinate (\x\y) ;
\filldraw[black] ({(\x + (\y / 2))*a},{(sqrt(3) * \y / 2)*a}) circle (1.5pt) ;
}

\draw[dashed] (03) -- (04) ;
\draw[dashed] (04) -- (13) ;
\draw[dashed] (03) -- (13) ;
\draw[dashed] (03) -- (12) ;
\draw[dashed] (13) -- (12) ;
\draw[dashed] (02) -- (03) ;
\draw[dashed] (22) -- (13) ;
\draw[dashed] (02) -- (12) ;
\draw[dashed] (12) -- (22) ;
\draw[dashed] (12) -- (11) ;
\draw[dashed] (12) -- (21) ;
\draw[dashed] (11) -- (21) ;
\draw[dashed] (02) -- (11) ;
\draw[dashed] (22) -- (21) ;
\draw[dashed] (01) -- (02) ;
\draw[dashed] (31) -- (22) ;
\draw[dashed] (01) -- (11) ;
\draw[dashed] (21) -- (31) ;
\draw[dashed] (20) -- (11) ;
\draw[dashed] (20) -- (21) ;
\draw (00) -- (01) ;
\draw[dashed] (10) -- (11) ;
\draw (00) -- (10) ;
\draw (30) -- (40) ;
\draw[dashed] (21)-- (30) ;
\draw[dashed] (30) -- (31) ;
\draw[dashed] (31) -- (40) ;

\end{tikzpicture}
\end{center}

we must have $\{(2,1),(2,2)\} \in A$, otherwise vertex $(2,1)$ would be incident to an odd (namely 1) number of edges of $A$. Next we consider the edge $\{(3,1),(2,2)\}$. \\

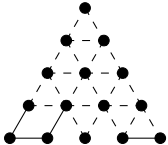
\begin{center}
\begin{tikzpicture}[declare function={a=0.5;}]
\foreach \y in {0,...,4} 
\foreach \x in {0,...,\numexpr 4 - \y}
{
\draw ({(\x + (\y / 2))*a},{(sqrt(3) * \y / 2)*a}) coordinate (\x\y) ;
\filldraw[black] ({(\x + (\y / 2))*a},{(sqrt(3) * \y / 2)*a}) circle (2pt) ;
}

\draw[dashed] (03) -- (04) ;
\draw[dashed] (04) -- (13) ;
\draw[dashed] (03) -- (13) ;
\draw[dashed] (03) -- (12) ;
\draw[dashed] (13) -- (12) ;
\draw[dashed] (02) -- (03) ;
\draw[dashed] (22) -- (13) ;
\draw[dashed] (02) -- (12) ;
\draw[dashed] (12) -- (22) ;
\draw[dashed] (12) -- (11) ;
\draw[dashed] (12) -- (21) ;
\draw[dashed] (11) -- (21) ;
\draw[dashed] (02) -- (11) ;
\draw[dashed] (22) -- (21) ;
\draw[dashed] (01) -- (02) ;
\draw[dashed] (31) -- (22) ;
\draw[dashed] (01) -- (11) ;
\draw[dashed] (21) -- (31) ;
\draw[dashed] (20) -- (11) ;
\draw[dashed] (20) -- (21) ;
\draw (00) -- (01) ;
\draw (10) -- (11) ;
\draw (00) -- (10) ;
\draw (30) -- (40) ;
\draw[dashed] (21)-- (30) ;
\draw[dashed] (30) -- (31) ;
\draw[dashed] (31) -- (40) ;

\end{tikzpicture}
\end{center}

We must have $\{(3,1),(2,2)\} \in A$, otherwise the face with vertices $(2,1)$, $(3,1),(2,2)$ will have an odd number of edges of $A$. \\

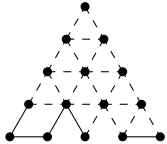
\begin{center}
\begin{tikzpicture}[declare function={a=0.5;}]
\foreach \y in {0,...,4} 
\foreach \x in {0,...,\numexpr 4 - \y}
{
\draw ({(\x + (\y / 2))*a},{(sqrt(3) * \y / 2)*a}) coordinate (\x\y) ;
\filldraw[black] ({(\x + (\y / 2))*a},{(sqrt(3) * \y / 2)*a}) circle (1.5pt) ;
}

\draw[dashed] (03) -- (04) ;
\draw[dashed] (04) -- (13) ;
\draw[dashed] (03) -- (13) ;
\draw[dashed] (03) -- (12) ;
\draw[dashed] (13) -- (12) ;
\draw[dashed] (02) -- (03) ;
\draw[dashed] (22) -- (13) ;
\draw[dashed] (02) -- (12) ;
\draw[dashed] (12) -- (22) ;
\draw[dashed] (12) -- (11) ;
\draw[dashed] (12) -- (21) ;
\draw[dashed] (11) -- (21) ;
\draw[dashed] (02) -- (11) ;
\draw[dashed] (22) -- (21) ;
\draw[dashed] (01) -- (02) ;
\draw[dashed] (31) -- (22) ;
\draw[dashed] (01) -- (11) ;
\draw[dashed] (21) -- (31) ;
\draw (20) -- (11) ;
\draw[dashed] (20) -- (21) ;
\draw (00) -- (01) ;
\draw (10) -- (11) ;
\draw (00) -- (10) ;
\draw (30) -- (40) ;
\draw[dashed] (21)-- (30) ;
\draw[dashed] (30) -- (31) ;
\draw[dashed] (31) -- (40) ;

\end{tikzpicture}
\end{center}

We fill in the right side by the symmetry about the middle established by Theorem 1. \\

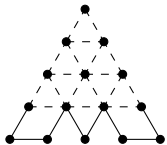
\begin{center}
\begin{tikzpicture}[declare function={a=0.5;}]
\foreach \y in {0,...,4} 
\foreach \x in {0,...,\numexpr 4 - \y}
{
\draw ({(\x + (\y / 2))*a},{(sqrt(3) * \y / 2)*a}) coordinate (\x\y) ;
\filldraw[black] ({(\x + (\y / 2))*a},{(sqrt(3) * \y / 2)*a}) circle (1.5pt) ;
}

\draw[dashed] (03) -- (04) ;
\draw[dashed] (04) -- (13) ;
\draw[dashed] (03) -- (13) ;
\draw[dashed] (03) -- (12) ;
\draw[dashed] (13) -- (12) ;
\draw[dashed] (02) -- (03) ;
\draw[dashed] (22) -- (13) ;
\draw[dashed] (02) -- (12) ;
\draw[dashed] (12) -- (22) ;
\draw[dashed] (12) -- (11) ;
\draw[dashed] (12) -- (21) ;
\draw[dashed] (11) -- (21) ;
\draw[dashed] (02) -- (11) ;
\draw[dashed] (22) -- (21) ;
\draw[dashed] (01) -- (02) ;
\draw[dashed] (31) -- (22) ;
\draw[dashed] (01) -- (11) ;
\draw[dashed] (21) -- (31) ;
\draw (20) -- (11) ;
\draw (20) -- (21) ;
\draw (00) -- (01) ;
\draw (10) -- (11) ;
\draw (00) -- (10) ;
\draw (30) -- (40) ;
\draw (21)-- (30) ;
\draw (31) -- (40) ;
\end{tikzpicture}
\end{center}

 We must have $\{(1,2),(2,2)\} \in A$ otherwise the face with vertices $(2,1)$, $(1,2),(2,2)$ will contain an odd number of edges of $A$. By symmetry, this must mean $\{(3,2),(4,2)\} \in A$ as well. Finally, we have $\{(2,2),(3,2)\} \not\in A$ because otherwise the face with vertices $(3,1), (2,2), (3,2)$ will have an odd number of edges in $A$. \\

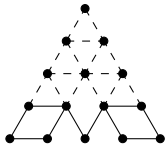
\begin{center}
\begin{tikzpicture}[declare function={a=0.5;}]
\foreach \y in {0,...,4} 
\foreach \x in {0,...,\numexpr 4 - \y}
{
\draw ({(\x + (\y / 2))*a},{(sqrt(3) * \y / 2)*a}) coordinate (\x\y) ;
\filldraw[black] ({(\x + (\y / 2))*a},{(sqrt(3) * \y / 2)*a}) circle (1.5pt) ;
}

\draw[dashed] (03) -- (04) ;
\draw[dashed] (04) -- (13) ;
\draw[dashed] (03) -- (13) ;
\draw[dashed] (03) -- (12) ;
\draw[dashed] (13) -- (12) ;
\draw[dashed] (02) -- (03) ;
\draw[dashed] (22) -- (13) ;
\draw[dashed] (02) -- (12) ;
\draw[dashed] (12) -- (22) ;
\draw[dashed] (12) -- (11) ;
\draw[dashed] (12) -- (21) ;
\draw[dashed] (02) -- (11) ;
\draw[dashed] (22) -- (21) ;
\draw[dashed] (01) -- (02) ;
\draw[dashed] (31) -- (22) ;
\draw (01) -- (11) ;
\draw (21) -- (31) ;
\draw (20) -- (11) ;
\draw (20) -- (21) ;
\draw (00) -- (01) ;
\draw (10) -- (11) ;
\draw (00) -- (10) ;
\draw (30) -- (40) ;
\draw (21)-- (30) ;
\draw (31) -- (40) ;
\end{tikzpicture}
\end{center}

We must have $\{(1,2),(1,3)\} \not\in A$ otherwise vertex $(1,2)$ will be incident to an odd number of edges of $A$. By symmetry, $\{(4,2),(3,3)\} \not\in A$. \\

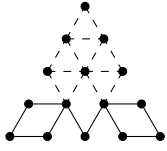
\begin{center}
\begin{tikzpicture}[declare function={a=0.5;}]
\foreach \y in {0,...,4} 
\foreach \x in {0,...,\numexpr 4 - \y}
{
\draw ({(\x + (\y / 2))*a},{(sqrt(3) * \y / 2)*a}) coordinate (\x\y) ;
\filldraw[black] ({(\x + (\y / 2))*a},{(sqrt(3) * \y / 2)*a}) circle (1.5pt) ;
}

\draw[dashed] (03) -- (04) ;
\draw[dashed] (04) -- (13) ;
\draw[dashed] (03) -- (13) ;
\draw[dashed] (03) -- (12) ;
\draw[dashed] (13) -- (12) ;
\draw[dashed] (02) -- (03) ;
\draw[dashed] (22) -- (13) ;
\draw[dashed] (02) -- (12) ;
\draw[dashed] (12) -- (22) ;
\draw[dashed] (12) -- (11) ;
\draw[dashed] (12) -- (21) ;
\draw[dashed] (02) -- (11) ;
\draw[dashed] (22) -- (21) ;
\draw (01) -- (11) ;
\draw (21) -- (31) ;
\draw (20) -- (11) ;
\draw (20) -- (21) ;
\draw (00) -- (01) ;
\draw (10) -- (11) ;
\draw (00) -- (10) ;
\draw (30) -- (40) ;
\draw (21)-- (30) ;
\draw (31) -- (40) ;
\end{tikzpicture}
\end{center}

We must have $\{(1,3),(2,2)\} \in A$, otherwise the face with vertices $(1,2),(2,2),(1,3)$ will contain an odd number of edges of $A$. By symmetry, $\{(3,2),(3,3)\} \in A$. \\ 

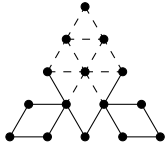
\begin{center}
\begin{tikzpicture}[declare function={a=0.5;}]
\foreach \y in {0,...,4} 
\foreach \x in {0,...,\numexpr 4 - \y}
{
\draw ({(\x + (\y / 2))*a},{(sqrt(3) * \y / 2)*a}) coordinate (\x\y) ;
\filldraw[black] ({(\x + (\y / 2))*a},{(sqrt(3) * \y / 2)*a}) circle (1.5pt) ;
}

\draw[dashed] (03) -- (04) ;
\draw[dashed] (04) -- (13) ;
\draw[dashed] (03) -- (13) ;
\draw[dashed] (03) -- (12) ;
\draw[dashed] (13) -- (12) ;
\draw[dashed] (02) -- (03) ;
\draw[dashed] (22) -- (13) ;
\draw[dashed] (02) -- (12) ;
\draw[dashed] (12) -- (22) ;
\draw[dashed] (12) -- (11) ;
\draw[dashed] (12) -- (21) ;
\draw (02) -- (11) ;
\draw (22) -- (21) ;
\draw (01) -- (11) ;
\draw (21) -- (31) ;
\draw (20) -- (11) ;
\draw (20) -- (21) ;
\draw (00) -- (01) ;
\draw (10) -- (11) ;
\draw (00) -- (10) ;
\draw (30) -- (40) ;
\draw (21)-- (30) ;
\draw (31) -- (40) ;
\end{tikzpicture}
\end{center}

We must have $\{(2,2),(2,3)\} \not\in A$ otherwise vertex $(2,2)$ will be incident to an odd number of edges. By symmetry, $\{(3,2),(2,3)\} \not\in A$ as well. \\

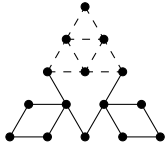
\begin{center}
\begin{tikzpicture}[declare function={a=0.5;}]
\foreach \y in {0,...,4} 
\foreach \x in {0,...,\numexpr 4 - \y}
{
\draw ({(\x + (\y / 2))*a},{(sqrt(3) * \y / 2)*a}) coordinate (\x\y) ;
\filldraw[black] ({(\x + (\y / 2))*a},{(sqrt(3) * \y / 2)*a}) circle (1.5pt) ;
}

\draw[dashed] (03) -- (04) ;
\draw[dashed] (04) -- (13) ;
\draw[dashed] (03) -- (13) ;
\draw[dashed] (03) -- (12) ;
\draw[dashed] (13) -- (12) ;
\draw[dashed] (02) -- (03) ;
\draw[dashed] (22) -- (13) ;
\draw[dashed] (02) -- (12) ;
\draw[dashed] (12) -- (22) ;
\draw (02) -- (11) ;
\draw (22) -- (21) ;
\draw (01) -- (11) ;
\draw (21) -- (31) ;
\draw (20) -- (11) ;
\draw (20) -- (21) ;
\draw (00) -- (01) ;
\draw (10) -- (11) ;
\draw (00) -- (10) ;
\draw (30) -- (40) ;
\draw (21)-- (30) ;
\draw (31) -- (40) ;
\end{tikzpicture}
\end{center}

We must have $\{(1,3),(2,3)\} \not \in A$ otherwise the face with vertices $(2,2),(1,3)$, $(2,3)$ will have an odd number of edges in $A$. By symmetry, $\{(2,3),(3,3)\} \not\in A$ as well. \\

In the general case for any totally even subset of $T_n$, one can continue along in this manner by working from the bottom to the top, by considering edges of the form $\{(x,y),(x,y+1)\}$ and $\{(x+1,y), (x,y+1)\}$ as we scan from left to right, and then we go up and consider edges of the form $\{(x,y),(x+1,y)\}$ as we scan from left to right, and then we go up and consider edges of the form $\{(x,y),(x,y+1)\}$ and $\{(x+1,y), (x,y+1)\}$ again and alternate. There will always exist some vertex or face with one remaining edge left that we are still considering, while we already know all its other edges to be in $A$ or not in $A$. 
\end{proof}

\begin{corollary}
Let $A$ and $B$ be totally even subsets of $T_n$. If $A$ and $B$ agree on a side (left, bottom, or right), then $A = B$.
\end{corollary}

\begin{proof}
consider the proof in Theorem 2.
\end{proof}

We claim the three totally even subsets in Figure 4 (located at the beginning of Section 2) formed a basis for all totally even subsets of $T_6$. Consider the left half of the bottom side, in particular consider the edges $\{(1,1),(2,1)\}$, $\{(2,1),(3,1)\},\{(3,1),(4,1)\}$. We first note the three totally even subsets in Figure 4 contain exactly one edge of $\{(1,1),(2,1)\},\{(2,1),(3,1)\},\{(3,1),(4,1)\}$. In this sense, they correspond to the standard basis vectors $(1,0,0), (0,1,0),(0,0,1)$. We can't have a larger basis because by Theorem 2 the bottom side completely determines a totally even subset, and furthermore by Theorem 1 the left half of the bottom side completely determines the bottom side of a totally even subset by symmetry. The next theorem will be to construct these standard basis vectors. We first define left half of the bottom side for general $T_n$. \\

\begin{definition}
The \textbf{left half of the bottom side} of $T_n$ are the edges $\{(i,1),(i+1)\}$ for $i \in \left[ \lfloor \frac{n}{2} \rfloor \right]$ and vertices $(i,1)$ for $i \in \left[ \lceil \frac{n+1}{2} \rceil \right]$. 
\end{definition}

\begin{theorem}
For each $i \in \left[ \lfloor \frac{n}{2} \rfloor \right]$, there exists a totally even subset $A$ of $T_n$ such that $\{(i,1),(i+1,1)\} \in A$ and $\{(i',1),(i'+1,1)\} \not\in A$ for all $i' \in \left[ \lfloor \frac{n}{2} \rfloor \right] \setminus \{i\}$. Furthermore, such indicator vectors form a basis of all totally even subsets. 
\end{theorem}

\begin{proof}
We first show given any $i \in \left[ \lfloor \frac{n}{2} \rfloor \right]$, there exists a totally even subset $A$ in $T_n$ such that \linebreak $\mathbbm{1}_A(\{(i,1),(i+1,1)\}) = 1$ and $\mathbbm{1}_A(\{(i',1),(i'+1,1)\}) = 0$ for all $i' \in \left[ \lfloor \frac{n}{2} \rfloor \right] \setminus \{i\}$. Define $E_1 \subseteq E(T_n)$ to be 

$$E_1 = \{\{(x,y),(x+1,y)\} \in E(T_n) : 1 \leq x \leq i, \ x+y \geq i+1, \ y \leq n+2-i \}$$

\noindent
and define $E_2 \subseteq E(T_n)$ to be 

$$E_2 = \{\{(x,y),(x-1,y+1)\} \in E(T_n) : 2 \leq x \leq i+1, \ x+y \geq i+1, \ y \leq n+1-i \}.$$

Define $A_1 = E_1 \cup E_2$, and define $A_2$ to be $A_1$ rotated 120 degrees counter-clockwise about the center of $T_n$, and define $A_3$ to be $A_2$ rotated 120 degrees counter-clockwise about the center of $T_n$. We give an example of $A_1,A_2,A_3$ with $i = 3$ in $T_{12}$ below. \\

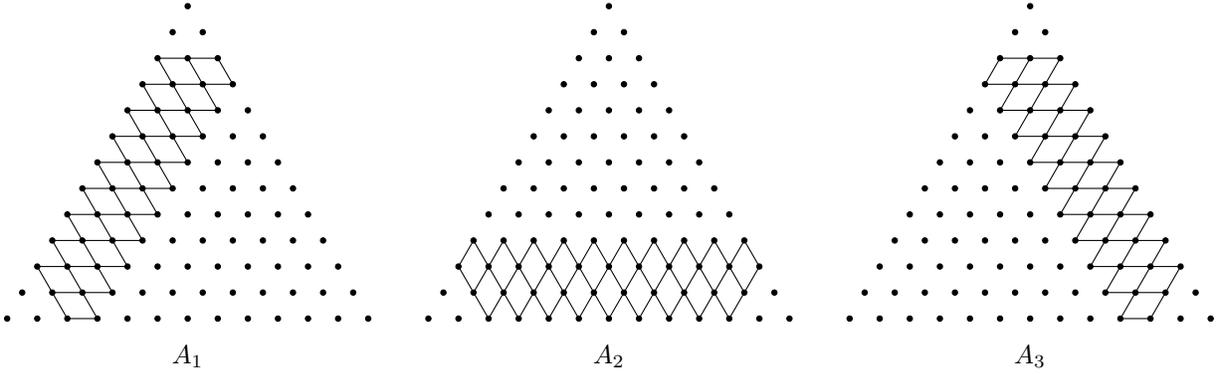
\begin{center}
\begin{tikzpicture}[declare function={a=0.4; f(\x,\y) = (\x + (\y / 2))*a ; g(\x,\y) =  (sqrt(3) * \y / 2)*a ; }]

\foreach \y in {0,...,12} 
\foreach \x in {0,...,\numexpr 12 - \y}
{
\filldraw[black] ({f(\x,\y)},{g(\x,\y)}) circle (1pt) ;
}

\foreach \y in {2,...,10}
{
\draw ({((\y / 2))*a},{(sqrt(3) * \y / 2)*a}) -- ({(1 + (\y / 2))*a},{(sqrt(3) * \y / 2)*a}) ;
}

\foreach \y in {1,...,10}
{
\draw ({(1+(\y / 2))*a},{(sqrt(3) * \y / 2)*a}) -- ({(2 + (\y / 2))*a},{(sqrt(3) * \y / 2)*a}) ;
}

\foreach \y in {0,...,9}
{
\draw ({(2+(\y / 2))*a},{(sqrt(3) * \y / 2)*a}) -- ({(3 + (\y / 2))*a},{(sqrt(3) * \y / 2)*a}) ;
}

\foreach \y in {1,...,9}
{
\draw ({(1+(\y / 2))*a},{(sqrt(3) * \y / 2)*a}) -- ({(((\y + 1) / 2))*a},{(sqrt(3) * (\y + 1) / 2)*a}) ;
}

\foreach \y in {0,...,9}
{
\draw ({(2+(\y / 2))*a},{(sqrt(3) * \y / 2)*a}) -- ({(1+((\y + 1) / 2))*a},{(sqrt(3) * (\y + 1) / 2)*a}) ;
}

\foreach \y in {0,...,9}
{
\draw ({(3+(\y / 2))*a},{(sqrt(3) * \y / 2)*a}) -- ({(2+((\y + 1) / 2))*a},{(sqrt(3) * (\y + 1) / 2)*a}) ;
}

\foreach \y in {0,...,12} 
\foreach \x in {0,...,\numexpr 12 - \y}
{
\filldraw[black] ({(\x + (\y / 2))*a + 14*a},{(sqrt(3) * \y / 2)*a}) circle (1pt) ;
}

\foreach \x in {2,...,10}
{
\draw ({f(\x,0) + 14*a},{g(\x,0)}) -- ({f(\x,1) + 14*a},{g(\x,1)}) ;
}

\foreach \x in {1,...,10}
{
\draw ({f(\x,1)+14*a},{g(\x,1)}) -- ({f(\x,2)+14*a},{g(\x,2)}) ;
}

\foreach \x in {0,...,9}
{
\draw ({f(\x,2)+14*a},{g(\x,2)}) -- ({f(\x,3) + 14*a},{g(\x,3)}) ;
}

\foreach \x in {2,...,10}
{
\draw ({f(\x,0) + 14*a},{g(\x,0)}) -- ({f(\x-1,1) + 14*a},{g(\x-1,1)}) ;
}

\foreach \x in {1,...,10}
{
\draw ({f(\x,1)+14*a},{g(\x,1)}) -- ({f(\x-1,2)+14*a},{g(\x-1,2)}) ;
}

\foreach \x in {1,...,10}
{
\draw ({f(\x,2)+14*a},{g(\x,2)}) -- ({f(\x-1,3) + 14*a},{g(\x-1,3)}) ;
}

\foreach \y in {0,...,12} 
\foreach \x in {0,...,\numexpr 12 - \y}
{
\filldraw[black] ({(\x + (\y / 2))*a + 28*a},{(sqrt(3) * \y / 2)*a}) circle (1pt) ;
}

\foreach \y in {0,...,9}
{
\draw ({f(9-\y,\y) + 28*a},{g(9-\y,\y)}) -- ({f(10-\y,\y) + 28*a},{g(10-\y,\y)}) ;
}

\foreach \y in {1,...,10}
{
\draw ({f(10-\y,\y) + 28*a},{g(10-\y,\y)}) -- ({f(11-\y,\y) + 28*a},{g(11-\y,\y)}) ;
}

\foreach \y in {2,...,10}
{
\draw ({f(11-\y,\y) + 28*a},{g(11-\y,\y)}) -- ({f(12-\y,\y) + 28*a},{g(12-\y,\y)}) ;
}

\foreach \y in {0,...,9}
{
\draw ({f(9-\y,\y) + 28*a},{g(9-\y,\y)}) -- ({f(9-\y,\y+1) + 28*a},{g(9-\y,\y+1)}) ;
}

\foreach \y in {0,...,9}
{
\draw ({f(10-\y,\y) + 28*a},{g(10-\y,\y)}) -- ({f(10-\y,\y+1) + 28*a},{g(10-\y,\y+1)}) ;
}

\foreach \y in {1,...,9}
{
\draw ({f(11-\y,\y) + 28*a},{g(11-\y,\y)}) -- ({f(11-\y,\y+1) + 28*a},{g(11-\y,\y+1)}) ;
}

\draw (6*a,0) node[yshift=-0.5cm]{$A_1$} ;

\draw (20*a,0) node[yshift=-0.5cm]{$A_2$} ;

\draw (34*a,0) node[yshift=-0.5cm]{$A_3$} ;
\end{tikzpicture}
\end{center}

We define $A = A_1 \triangle A_2 \triangle A_3$. Note $\{(i,1),(i+1,1)\} \in A$. We claim $A$ is a totally even subset. We first provide an equivalent but more useful way to describe $A_1$. Consider the equilateral triangle $R_1$ with points $(i+1,0),(i+1,n+2-i),(2i-n-1,n+2-i)$ in the coordinate system shown in Figure 3 (located in the Introduction). Let $A_1$ be the set of edges of the form $\{(x,y),(x+1,y)\}$ or $\{(x,y),(x-1,y+1)\}$ contained inside $R_1$, including its boundary. We show a picture of $A_1$ when $i = 3$ in $T_{12}$, and $R_1$ (dashed lines). 

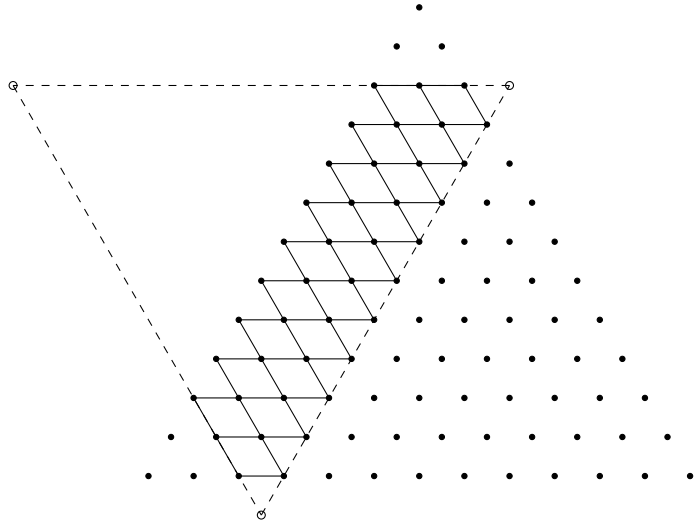
\begin{center}
\begin{tikzpicture}[declare function={a=0.6; f(\x,\y) = (\x + (\y / 2))*a ; g(\x,\y) =  (sqrt(3) * \y / 2)*a ; }]

\foreach \y in {0,...,12} 
\foreach \x in {0,...,\numexpr 12 - \y}
{
\filldraw[black] ({f(\x,\y)},{g(\x,\y)}) circle (1pt) ;
}

\foreach \y in {2,...,10}
{
\draw ({((\y / 2))*a},{(sqrt(3) * \y / 2)*a}) -- ({(1 + (\y / 2))*a},{(sqrt(3) * \y / 2)*a}) ;
}

\foreach \y in {1,...,10}
{
\draw ({(1+(\y / 2))*a},{(sqrt(3) * \y / 2)*a}) -- ({(2 + (\y / 2))*a},{(sqrt(3) * \y / 2)*a}) ;
}

\foreach \y in {0,...,9}
{
\draw ({(2+(\y / 2))*a},{(sqrt(3) * \y / 2)*a}) -- ({(3 + (\y / 2))*a},{(sqrt(3) * \y / 2)*a}) ;
}

\foreach \y in {1,...,9}
{
\draw ({(1+(\y / 2))*a},{(sqrt(3) * \y / 2)*a}) -- ({(((\y + 1) / 2))*a},{(sqrt(3) * (\y + 1) / 2)*a}) ;
}

\foreach \y in {0,...,9}
{
\draw ({(2+(\y / 2))*a},{(sqrt(3) * \y / 2)*a}) -- ({(1+((\y + 1) / 2))*a},{(sqrt(3) * (\y + 1) / 2)*a}) ;
}

\foreach \y in {0,...,9}
{
\draw ({(3+(\y / 2))*a},{(sqrt(3) * \y / 2)*a}) -- ({(2+((\y + 1) / 2))*a},{(sqrt(3) * (\y + 1) / 2)*a}) ;
} 

\draw ({f(3,-1)},{g(3,-1)}) circle (1.5pt) ;
\draw ({f(3,10)},{g(3,10)}) circle (1.5pt) ;
\draw ({f(-8,10)},{g(-8,10)}) circle (1.5pt) ;

\draw[dashed] ({f(3,-1)},{g(3,-1)}) -- ({f(3,10)},{g(3,10)}) ;
\draw[dashed] ({f(3,10)},{g(3,10)}) -- ({f(-8,10)},{g(-8,10)}) ;
\draw[dashed] ({f(-8,10)},{g(-8,10)}) -- ({f(3,-1)},{g(3,-1)}) ;

\end{tikzpicture}
\end{center}

Let $R_2$ and $R_3$ be the analogous traingular regions corresponding to $A_2$ and $A_3$ respectively. So $A$ looks like

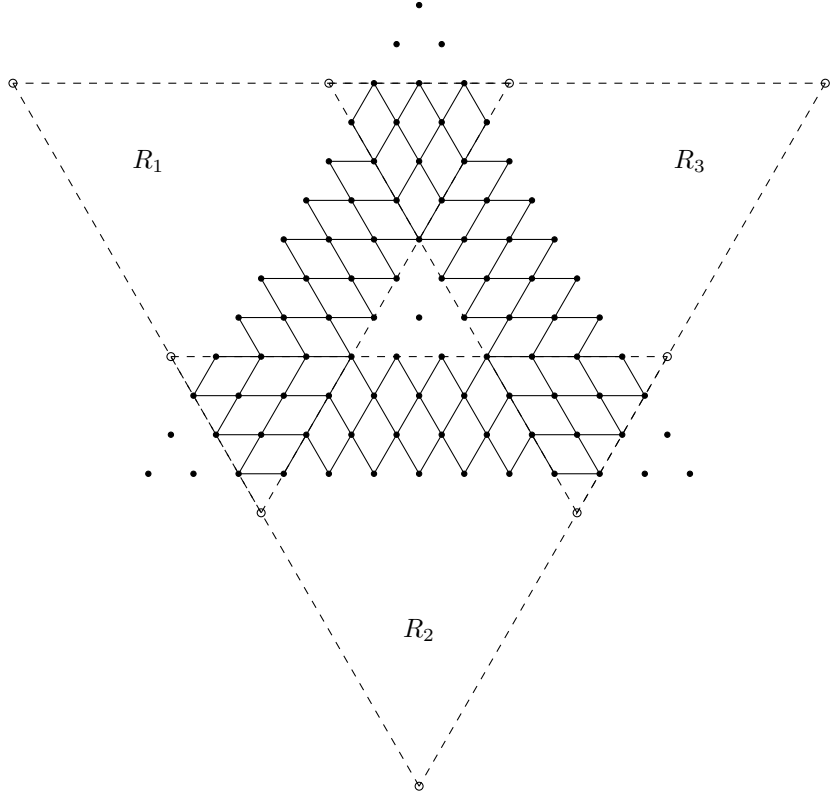
\begin{center}
\begin{tikzpicture}[declare function={a=0.6; f(\x,\y) = (\x + (\y / 2))*a ; g(\x,\y) =  (sqrt(3) * \y / 2)*a ; }]

\foreach \y in {0,...,12} 
\foreach \x in {0,...,\numexpr 12 - \y}
{
\filldraw[black] ({f(\x,\y)},{g(\x,\y)}) circle (1pt) ;
}

\foreach \y in {2,...,8}
{
\draw ({((\y / 2))*a},{(sqrt(3) * \y / 2)*a}) -- ({(1 + (\y / 2))*a},{(sqrt(3) * \y / 2)*a}) ;
}

\foreach \y in {1,...,7}
{
\draw ({(1+(\y / 2))*a},{(sqrt(3) * \y / 2)*a}) -- ({(2 + (\y / 2))*a},{(sqrt(3) * \y / 2)*a}) ;
}

\foreach \y in {0,...,6}
{
\draw ({(2+(\y / 2))*a},{(sqrt(3) * \y / 2)*a}) -- ({(3 + (\y / 2))*a},{(sqrt(3) * \y / 2)*a}) ;
}

\foreach \y in {3,...,9}
{
\draw ({(1+(\y / 2))*a},{(sqrt(3) * \y / 2)*a}) -- ({(((\y + 1) / 2))*a},{(sqrt(3) * (\y + 1) / 2)*a}) ;
}

\foreach \y in {3,...,9}
{
\draw ({(2+(\y / 2))*a},{(sqrt(3) * \y / 2)*a}) -- ({(1+((\y + 1) / 2))*a},{(sqrt(3) * (\y + 1) / 2)*a}) ;
}

\foreach \y in {3,...,9}
{
\draw ({(3+(\y / 2))*a},{(sqrt(3) * \y / 2)*a}) -- ({(2+((\y + 1) / 2))*a},{(sqrt(3) * (\y + 1) / 2)*a}) ;
} 

\draw ({f(3,-1)},{g(3,-1)}) circle (1.5pt) ;
\draw ({f(3,10)},{g(3,10)}) circle (1.5pt) ;
\draw ({f(-8,10)},{g(-8,10)}) circle (1.5pt) ;

\draw[dashed] ({f(3,-1)},{g(3,-1)}) -- ({f(3,10)},{g(3,10)}) ;
\draw[dashed] ({f(3,10)},{g(3,10)}) -- ({f(-8,10)},{g(-8,10)}) ;
\draw[dashed] ({f(-8,10)},{g(-8,10)}) -- ({f(3,-1)},{g(3,-1)}) ;

\foreach \x in {2,...,8}
{
\draw ({f(\x,0)},{g(\x,0)}) -- ({f(\x,1)},{g(\x,1)}) ;
}

\foreach \x in {1,...,7}
{
\draw ({f(\x,1)},{g(\x,1)}) -- ({f(\x,2)},{g(\x,2)}) ;
}

\foreach \x in {0,...,6}
{
\draw ({f(\x,2)},{g(\x,2)}) -- ({f(\x,3)},{g(\x,3)}) ;
}

\foreach \x in {4,...,10}
{
\draw ({f(\x,0)},{g(\x,0)}) -- ({f(\x-1,1)},{g(\x-1,1)}) ;
}

\foreach \x in {4,...,10}
{
\draw ({f(\x,1)},{g(\x,1)}) -- ({f(\x-1,2)},{g(\x-1,2)}) ;
}

\foreach \x in {4,...,10}
{
\draw ({f(\x,2)},{g(\x,2)}) -- ({f(\x-1,3)},{g(\x-1,3)}) ;
}

\draw ({f(10,-1)},{g(10,-1)}) circle (1.5pt) ;
\draw ({f(-1,10)},{g(-1,10)}) circle (1.5pt) ;
\draw ({f(10,10)},{g(10,10)}) circle (1.5pt) ;

\draw[dashed] ({f(10,-1)},{g(10,-1)}) -- ({f(-1,10)},{g(-1,10)}) ;
\draw[dashed] ({f(-1,10)},{g(-1,10)}) -- ({f(10,10)},{g(10,10)}) ;
\draw[dashed] ({f(10,10)},{g(10,10)}) -- ({f(10,-1)},{g(10,-1)}) ;

\foreach \y in {0,...,6}
{
\draw ({f(9-\y,\y)},{g(9-\y,\y)}) -- ({f(10-\y,\y)},{g(10-\y,\y)}) ;
}

\foreach \y in {1,...,7}
{
\draw ({f(10-\y,\y)},{g(10-\y,\y)}) -- ({f(11-\y,\y)},{g(11-\y,\y)}) ;
}

\foreach \y in {2,...,8}
{
\draw ({f(11-\y,\y)},{g(11-\y,\y)}) -- ({f(12-\y,\y)},{g(12-\y,\y)}) ;
}

\foreach \y in {3,...,9}
{
\draw ({f(9-\y,\y)},{g(9-\y,\y)}) -- ({f(9-\y,\y+1)},{g(9-\y,\y+1)}) ;
}

\foreach \y in {3,...,9}
{
\draw ({f(10-\y,\y)},{g(10-\y,\y)}) -- ({f(10-\y,\y+1)},{g(10-\y,\y+1)}) ;
}

\foreach \y in {3,...,9}
{
\draw ({f(11-\y,\y)},{g(11-\y,\y)}) -- ({f(11-\y,\y+1)},{g(11-\y,\y+1)}) ;
}

\draw ({f(-1,3)},{g(-1,3)}) circle (1.5pt) ;
\draw ({f(10,3)},{g(10,3)}) circle (1.5pt) ;
\draw ({f(10,-8)},{g(10,-8)}) circle (1.5pt) ;

\draw[dashed] ({f(-1,3)},{g(-1,3)}) -- ({f(10,3)},{g(10,3)}) ;
\draw[dashed] ({f(10,3)},{g(10,3)}) -- ({f(10,-8)},{g(10,-8)}) ;
\draw[dashed] ({f(10,-8)},{g(10,-8)}) -- ({f(-1,3)},{g(-1,3)}) ;

\draw ({f(-4,8)},{g(-4,8)}) node{$R_1$} ;
\draw ({f(8,8)},{g(8,8)}) node{$R_3$} ;
\draw ({f(8,-4)},{g(8,-4)}) node{$R_2$} ;

\end{tikzpicture}
\end{center}

In general, there will be $7$ regions of interest: $R_1 \setminus (R_2 \cup R_3)$, $R_2 \setminus (R_1 \cup R_3)$, $R_3 \setminus (R_1 \cup R_2)$, $R_1 \cap R_2$, $R_2 \cap R_3$, $R_3 \cap R_1$, and $R_1 \cap R_2 \cap R_3$. Note the region $R_1 \cap R_2 \cap R_3$ will be non-empty when $i \geq \lceil \frac{n}{2} \rceil$, however it will never contain any edges because the symmetric difference between $A_1,A_2,A_3$ in the region $R_1 \cap R_2 \cap R_3$ cancels out completely, since $A_1$ contains edges of the form $\{(x,y),(x+1,y)\},\{(x,y),(x-1,y+1)\}$ and $A_2$ contains edges of the form $\{(x,y),(x-1,y+1)\},\{(x,y),(x,y+1)\}$ and $A_3$ contains edges of the form $\{(x,y),(x,y+1)\},\{(x,y),(x+1,y)\}$. Likewise, the regions $R_1 \setminus (R_2 \cup R_3)$ and $R_2 \cap R_3$ contain edges of the form $\{(x,y),(x+1,y)\},\{(x,y),(x-1,y+1)\}$; the regions  $R_2 \setminus (R_1 \cup R_3)$ and $R_1 \cap R_3$ contain edges of the form $\{(x,y),(x-1,y+1)\},\{(x,y),(x,y+1)\}$; the regions $R_3 \setminus (R_1 \cup R_2)$ and $R_1 \cap R_2$ contain edges of the form $\{(x,y),(x,y+1)\},\{(x,y),(x+1,y)\}$. One can can check any vertex in each region contains an even number of edges of $A$, and every face has an even number of edges of $A$. \\

Thus given any $i \in \left[ \lfloor \frac{n}{2} \rfloor \right]$, there exists a totally even subset $A$ in $T_n$ such that $\mathbbm{1}_A(\{(i,1),(i+1,1)\}) = 1$ and $\mathbbm{1}_A(\{(i',1),(i'+1,1)\}) = 0$ for all $i' \in \left[ \lfloor \frac{n}{2} \rfloor \right] \setminus \{i\}$. Denote the set of such indicator vectors by $F$, so $|F| = \lfloor \frac{n}{2} \rfloor$. Note the vectors in $F$ are linearly independent since each vector in $F$ has a $1$ in an entry $\{(i,1),(i+1,1)\}$ for some $i \in \left[ \lfloor \frac{n}{2} \rfloor \right]$ while the other vectors have a $0$ at entry $\{(i,1),(i+1,1)\}$. Now we show they span the entire space of totally even subsets. By Theorem 2, the bottom side of a totally even subset uniquely determines it. By Theorem 1, a totally even subset is symmetric about the middle, and thus its bottom side is also symmetric about the middle, and thus the left half of the bottom side of a totally even subset uniquely determines it. We are done in the case $n$ is even, as you can create a totally even subset with any ``left half of the bottom side" using the vectors in $F$. In the case that $n$ is odd, by Corollary $1$ the edge at the bottom lying on the middle, $\{(\frac{n+1}{2},1),(\frac{n+1}{2}+1,1)\}$, cannot be in any totally even susbet, thus we may ignore this edge and only consider the first $\lfloor \frac{n}{2} \rfloor$ edges in the left half of the bottom side. 
\end{proof}

The preceding Theorem allows us to make the following definition and remark. 

\begin{definition}
For $i \in \left[ \lfloor \frac{n}{2} \rfloor \right]$, let $A(i)$ denote the totally even subset of $T_n$ containing exactly one edge of the bottom side of $T_n$, namely $\{(i,1),(i+1,1)\}$.
\end{definition}

\begin{remark}
Any totally even subset $A$ of $T_n$ can be decomposed into $A(i)$, that is, $A = A(i_1) \triangle A(i_2) \triangle \cdots \triangle A(i_k)$. From now on, when we write $A$ this way we also assume $i_s < i_t$ for all $s < t$ and $i_s \leq \frac{n}{2}$ for all $s \in [k]$.  
\end{remark}

\begin{center}
\section{Counting the Edges of Totally Even Subsets of $T_n$}
\end{center}

We first count the edges of the ``standard basis vectors," that is, totally even subsets that contain exactly one edge of the left half of the bottom side of $T_n$. \\

\begin{proposition}
In $T_n$, we have that $A(i)$ contains $6(n-2i+1)i$ edges. 
\end{proposition}

\begin{proof}
Recall in the proof of Theorem 3 we defined the edge sets $A_1,A_2,A_3$ and defined $A = A(i) = A_1 \triangle A_2 \triangle A_3$. Note, however, that $A_1,A_2,A_3$ cancelled out at some places, so to count the number of edges of $A(i)$ we can't just count the edges of $A_1,A_2,A_3$ and sum them. We describe another way to construct $A(i)$, in particular we will decompose $A(i) = B_1 \triangle B_2 \triangle B_3 \triangle B_4 \triangle B_5 \triangle B_6$ where the embeddings of $B_s,B_t$ are isometric for all $s,t \in [6]$, so in particular $|B_s| = |B_t|$ for all $s,t \in [6]$, and also $B_s \cap B_t = \{\}$ for all $s,t \in [6]$. \\

Define $B_1$ to be

$$B_1 = \{\{(x,y),(x,y+1)\} : 1 \leq y \leq i, \ i+1 \leq x + y \leq n-i\}.$$

\noindent
Another way to describe $B_1$ is to take all the edges of the form $\{(x,y),(x,y+1)\}$ in the parallelogram $R_1$ with points $i,1),(n-i+1,1),(0,i+1),(n-2i+1,i+1)$ in the coordinate system shown in Figure 3. We reuse the example of $T_{12}$ and $i = 3$ from Theorem 3. \\

\begin{center}
\begin{tikzpicture}[declare function={a=0.6; f(\x,\y) = (\x + (\y / 2))*a ; g(\x,\y) =  (sqrt(3) * \y / 2)*a ; }]

\foreach \y in {0,...,12} 
\foreach \x in {0,...,\numexpr 12 - \y}
{
\filldraw[black] ({f(\x,\y)},{g(\x,\y)}) circle (1pt) ;
}

\draw ({f(-1,3)},{g(-1,3)}) circle (1.5pt) ;
\draw ({f(6,3)},{g(6,3)}) circle (1.5pt) ;
\draw ({f(2,0)},{g(2,0)}) circle (1.5pt) ;
\draw ({f(9,0)},{g(9,0)}) circle (1.5pt) ;

\draw[dashed]  ({f(-1,3)},{g(-1,3)}) -- ({f(6,3)},{g(6,3)}) ;
\draw[dashed]  ({f(9,0)},{g(9,0)}) -- ({f(6,3)},{g(6,3)}) ;
\draw[dashed]  ({f(9,0)},{g(9,0)}) -- ({f(2,0)},{g(2,0)}) ;
\draw[dashed]  ({f(-1,3)},{g(-1,3)}) -- ({f(2,0)},{g(2,0)}) ;

\foreach \y in {0,...,2}
\foreach \x in {0,...,6}
{
\draw ({f({\x + 2 - \y},\y)},{g({\x + 2 - \y},\y)}) -- ({f({\x + 2 - \y},{\y + 1})},{g({\x + 2 - \y},{\y + 1})}) ;
}

\draw (6*a,-1) node{\Large $R_1$} ;

\end{tikzpicture}
\end{center}

Note if the length of an edge is $1$, then the product of the side lengths of two adjacent sides in the parallelogram $R_1$ is the number of edges it contains. Note $R_1$ has side length $n-2i+1$ at the bottom and length $i$ at the left, so it contains $(n-2i+1)i$ edges. We define $B_2$ to be the reflection of $B_1$ about the middle. Let $B_3$ be $B_1$ rotated 120 degrees counter-clockwise about the center of $T_n$, and $B_5$ be $B_3$ rotated 120 degrees counter-clockwise about the center of $T_n$. Likewise, $B_4$ is $B_2$ rotated 120 degrees counter-clockwise about the center of $T_n$, and $B_6$ is $B_4$ rotated 120 degrees counter-clockwise about the center of $T_n$. Note if $s \equiv t \mod 2$ and $s \neq t$, then $B_s \cap B_t = \{\}$ because $B_s$ and $B_t$ are rotations of each other and so the edges they contain will be ``going in different directions." For example, one may have edges of the form $\{(x,y),(x+1,y)\}$ and the other may have edges of the form $\{(x,y),(x,y+1)\}$ (note each $B_r$ for $r \in [6]$ has edges ``only going in one of the three directions"), so they can't possibly cancel out anywhere. We have $B_1 \cap B_2 = \{\}$ because $B_1$ and $B_2$ have edges going in different directions as they are reflections of each other. Likewise, $B_1 \cap B_6 = \{\}$ because $B_1$ and $B_6$ have edges going in different directions. If we let $R_s$ be the parallelogram region corresponding to $B_s$ for $s \in [6]$, note $B_1 \cap B_4 = \{\}$ because $R_1 \cap R_4$ intersects only at a single point. One can show $B_s \cap B_t = \{\}$ for other $s \neq t$ using analogous arguments. Note $|B_s| = |B_t|$ for all $s,t \in [6]$ by construction. One can check that $A_1 \triangle A_2 \triangle A_3 = B_1 \triangle B_2 \triangle B_3 \triangle B_4 \triangle B_5 \triangle B_6$, where $A_1,A_2,A_3$ are defined in Theorem 3. \\

Thus if $A(i) = A_1 \triangle A_2 \triangle A_3 = B_1 \triangle B_2 \triangle B_3 \triangle B_4 \triangle B_5 \triangle B_6$ and all the $B_s$ are pairwise disjoint, then $|A(i)| = \sum_{s \in [6]} |B_s| = 6|B_1| = 6(n-2i+1)i$, as required. 
\end{proof}

\begin{theorem}
Let $A$ be a totally even subset of $T_n$, and let $A = A(i_1) \triangle A(i_2) \triangle \cdots \triangle A(i_k)$ where $k > 0$. Furthermore, let $i_0 = 0$ and $i_{k+1} = \frac{n+1}{2}$. Let $a_s = i_{s+1} - i_s$ for $s \in \{0\} \cup [k]$. Then $A$ contains

$$12 \cdot \left(\sum_{s \in \{0\} \cup [k], \ s \text{ even}} a_s \right) \cdot \left(\sum_{s \in \{0\} \cup [k], \ s \text{ odd}} a_s \right)$$

\noindent
edges. 
\end{theorem}

\begin{proof}
We first begin with an example to illustrate a particular fact, and then using that fact we discuss the general case. Consider $T_{12}$ again, and let $A = A(2) \triangle A(3) \triangle A(6)$ so that $i_0 = 0, i_1 = 2, i_2 = 3, i_3 = 6, i_4 = \frac{13}{2}$. Thus $a_0 = 2, a_1 = 1, a_2 = 3, a_3 = \frac{1}{2}$. Note $|A| = |A(2) \triangle A(3) \triangle A(6)|$. \\

Recall $B_1$ defined in Proposition 1: it suffices to count the number of edges in the symmetric difference of the three $B_1$ corresponding to each of $A(2),A(3),A(6)$, and then multiply the result by 6 to get $|A(2) \triangle A(3) \triangle A(6)|$ (we will justify this later). To further simplify things even more, we can just consider the left half of the symmetric difference of the three $B_1$: consider the picture below. \\

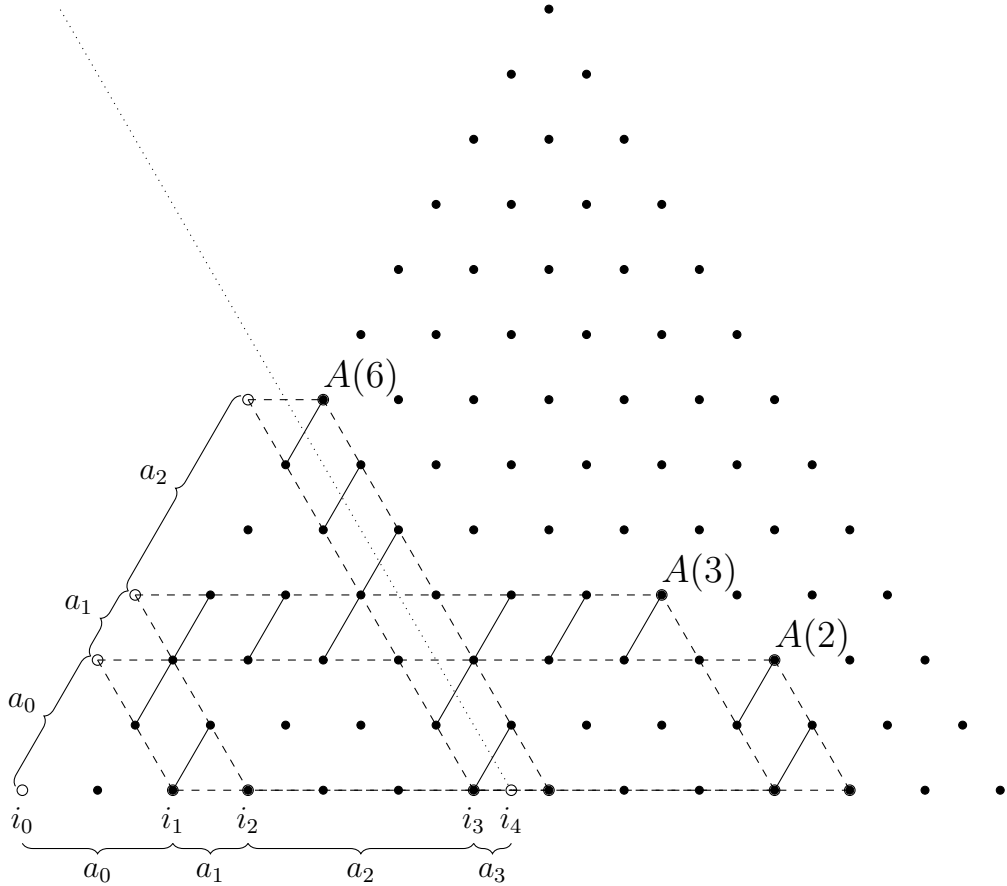
\begin{center}
\begin{tikzpicture}[declare function={a=1; f(\x,\y) = (\x + (\y / 2))*a ; g(\x,\y) =  (sqrt(3) * \y / 2)*a ; }]

\foreach \y in {0,...,12} 
\foreach \x in {0,...,\numexpr 12 - \y}
{
\draw ({f(\x,\y)},{g(\x,\y)}) coordinate (\x\y) ;
\filldraw[black] ({f(\x,\y)},{g(\x,\y)}) circle (1.5pt) ;
}

\draw ({f(-1,3)},{g(-1,3)}) circle (2pt) ;
\draw ({f(6,3)},{g(6,3)}) circle (2pt) ;
\draw ({f(2,0)},{g(2,0)}) circle (2pt) ;
\draw ({f(9,0)},{g(9,0)}) circle (2pt) ;

\draw[dashed]  ({f(-1,3)},{g(-1,3)}) -- ({f(6,3)},{g(6,3)}) ;
\draw[dashed]  ({f(9,0)},{g(9,0)}) -- ({f(6,3)},{g(6,3)}) ;
\draw[dashed]  ({f(9,0)},{g(9,0)}) -- ({f(2,0)},{g(2,0)}) ;
\draw[dashed]  ({f(-1,3)},{g(-1,3)}) -- ({f(2,0)},{g(2,0)}) ;

\draw ({f(-1,2)},{g(-1,2)}) circle (2pt) ;
\draw ({f(8,2)},{g(8,2)}) circle (2pt) ;
\draw ({f(1,0)},{g(1,0)}) circle (2pt) ;
\draw ({f(10,0)},{g(10,0)}) circle (2pt) ;

\draw[dashed]  ({f(-1,2)},{g(-1,2)}) -- ({f(8,2)},{g(8,2)}) ;
\draw[dashed]  ({f(10,0)},{g(10,0)}) -- ({f(8,2)},{g(8,2)}) ;
\draw[dashed]  ({f(10,0)},{g(10,0)}) -- ({f(1,0)},{g(1,0)}) ;
\draw[dashed]  ({f(-1,2)},{g(-1,2)}) -- ({f(1,0)},{g(1,0)}) ;

\draw ({f(-1,6)},{g(-1,6)}) circle (2pt) ;
\draw ({f(0,6)},{g(0,6)}) circle (2pt) ;
\draw ({f(5,0)},{g(5,0)}) circle (2pt) ;
\draw ({f(6,0)},{g(6,0)}) circle (2pt) ;

\draw[dashed]  ({f(-1,6)},{g(-1,6)}) -- ({f(0,6)},{g(0,6)}) ;
\draw[dashed]  ({f(6,0)},{g(6,0)}) -- ({f(0,6)},{g(0,6)}) ;
\draw[dashed]  ({f(6,0)},{g(6,0)}) -- ({f(5,0)},{g(5,0)}) ;
\draw[dashed]  ({f(-1,6)},{g(-1,6)}) -- ({f(5,0)},{g(5,0)}) ;

\draw (10) -- (11) ;
\draw (01) -- (02) ;
\draw (02) -- (03) ;
\draw (12) -- (13) ;
\draw (22) -- (23) ;
\draw (23) -- (24) ;
\draw (14) -- (15) ;
\draw (05) -- (06) ;
\draw (50) -- (51) ;
\draw (41) -- (42) ;
\draw (42) -- (43) ;
\draw (52) -- (53) ;
\draw (62) -- (63) ;
\draw (81) -- (82) ;
\draw (90) -- (91) ;

\draw[dotted] (5.5,0) -- ({f(-6.5,12)},{g(-6.5,12)}) ;

\draw ({f(-1,0)},{g(-1,0)}) coordinate (start) ;
\draw (start) circle (2pt) ;

\draw (5.5,0) coordinate (end) ;
\draw (end) circle (2pt) ;

\draw (start) node[yshift=-0.4cm]{\large $i_0$} ;
\draw (10) node[yshift=-0.4cm]{\large $i_1$} ;
\draw (20) node[yshift=-0.4cm]{\large $i_2$} ;
\draw (50) node[yshift=-0.4cm]{\large $i_3$} ;
\draw (end) node[yshift=-0.4cm]{\large $i_4$} ;

\draw[decorate,decoration={brace,raise=3pt,amplitude=4pt,mirror}] (-1,-0.6) -- (1,-0.6) node[midway,yshift=-0.5cm]{\large $a_0$} ;

\draw[decorate,decoration={brace,raise=3pt,amplitude=4pt,mirror}] (1,-0.6) -- (2,-0.6) node[midway,yshift=-0.5cm]{\large $a_1$} ;

\draw[decorate,decoration={brace,raise=3pt,amplitude=4pt,mirror}] (2,-0.6) -- (5,-0.6) node[midway,yshift=-0.5cm]{\large $a_2$} ;

\draw[decorate,decoration={brace,raise=3pt,amplitude=4pt,mirror}] (5,-0.6) -- (5.5,-0.6) node[midway,yshift=-0.5cm]{\large $a_3$} ;

\draw[decorate,decoration={brace,raise=3pt,amplitude=4pt}] (start) -- ({f(-1,2)},{g(-1,2)}) node[midway,xshift=-0.5cm,yshift=0.3cm]{\large $a_0$} ;

\draw[decorate,decoration={brace,raise=3pt,amplitude=4pt}] ({f(-1,2)},{g(-1,2)}) -- ({f(-1,3)},{g(-1,3)}) node[midway,xshift=-0.5cm,yshift=0.3cm]{\large $a_1$} ;

\draw[decorate,decoration={brace,raise=3pt,amplitude=4pt}] ({f(-1,3)},{g(-1,3)}) -- ({f(-1,6)},{g(-1,6)}) node[midway,xshift=-0.5cm,yshift=0.3cm]{\large $a_2$} ;

\draw (82) node[xshift=0.5cm,yshift=0.3cm]{\Large $A(2)$} ;
\draw (63) node[xshift=0.5cm,yshift=0.3cm]{\Large $A(3)$} ;
\draw (06) node[xshift=0.5cm,yshift=0.3cm]{\Large $A(6)$} ;
\end{tikzpicture}
\end{center}

As in Proposition 1, each $B_1$ corresponds to a parallelogram, and in the picture above I labelled at the top-right point of each parallelogram their corresponding $A(i)$ for $i \in \{2,3,6\}$. Recall in Proposition 1 that we counted the edges in a parallelogram by multiplying the side lengths of two adjacent sides of the parallelogram. Also note the parallelogram regions which contain edges are precisely those regions with an odd number of $B_1$ overlapping in those regions by definition of symmetric difference. Thus we get a checkerboard pattern, where the region corresponding to $a_sa_t$ for $s,t \in \{0\} \cup [k]$ has edges if and only if $s$ and $t$ have different parities. \\

In general, for $A = A(i_1) \triangle A(i_2) \triangle \ldots \triangle A(i_k)$ where $k > 0$, we take the symmetric difference of the $B_1$ corresponding to each $A(i_s)$ and count how many edges are in the ``left half" of the symmetric difference. By the example above, this number is

$$\left(\sum_{s \in \{0\} \cup [k], \ s \text{ even}} a_s \right) \cdot \left(\sum_{s \in \{0\} \cup [k], \ s \text{ odd}} a_s \right).$$

\noindent
Since this is only the left half of the symmetric difference of the $B_1$, the total number of edges in the symmetric difference of the $B_1$ is

$$2 \cdot \left(\sum_{s \in \{0\} \cup [k], \ s \text{ even}} a_s \right) \cdot \left(\sum_{s \in \{0\} \cup [k], \ s \text{ odd}} a_s \right).$$

\noindent

We do the same thing for $B_2,B_3,B_4,B_5,B_6$. We claim the symmetric difference of all $B_s$ and the symmetric difference of all $B_t$ are disjoint for $s \neq t$ and $s,t \in [6]$, which can be shown using similar arguments in Proposition 1 that argued $B_s \cap B_t = \emptyset$ for $s \neq t$ and $s,t \in [6]$. One argument in particular we want to elaborate a bit is the case when $s = 1$ and $t = 4$. In this case note that each $B_1$ lies below the line containing the vertex $(n+1,1)$ and the midpoint of vertices $(n,1),(n,2)$, while each $B_4$ lies above the line containing the vertex $(n+1,1)$ and the midpoint of vertices $(n,1),(n,2)$, and thus the symmetric difference of all $B_1$ and the symmetric of difference of all $B_4$ are disjoint. This way to count the total number of edges in $A$, we simply have to sum for $s \in [6]$ the number of edges in the symmetric difference of all $B_s$ (one $B_s$ for each $A(i_t)$ for $t \in [k]$). By symmetries, the number of edges in the symmetric difference of all $B_s$ equals the number of edges in the symmetric difference of all $B_1$. Thus we get that the number of edges in $A$ is 

$$12 \cdot \left(\sum_{s \in \{0\} \cup [k], \ s \text{ even}} a_s \right) \cdot \left(\sum_{s \in \{0\} \cup [k], \ s \text{ odd}} a_s \right)$$

\noindent
as required.
\end{proof}

\begin{center}
\section{Transversals}
\end{center}
We begin with an idea of Beluhov \cite{bel}. Let $G$ be a plane graph. Let $C_1,C_2$ be the edge sets of two cycles in $G$ with the same signature. As shown in the introduction, $A = C_1 \triangle C_2$ is a totally even subset. Note that each face in $G$ contains the same number of edges in $A \cap C_1$ as in $A \cap C_2$ (that this holds true for the infinite face, that is, the outside face, requires some work to show. See \cite{bel}). For each finite face in $G$, pair up the edges it contains in $A \cap C_1$ with the edges it contains in $A \cap C_2$ (which is possible by the preceding sentence), and for each pair of edges, connect them with a curve that begins at the midpoint of one edge and ends at the midpoint of the other. Do this in such a way such that the curves remain strictly inside the face, and doesn't intersect with any other curve. The following procedure shows this is possible: note that deleting any two edges of a face will split the face into two paths (the edges do not include the vertices, so that if we delete two incident edges, the vertex at which they are incident does not get deleted and will count as a path of length $0$). Find one edge in $A \cap C_1$ and one edge in $A \cap C_2$ such that when they are deleted, for one of the two paths, there is no edge in $A$ on it. Pair these two edges up, and now treat them as if they were not in $A$, and repeat the process (find another pair of edges so that no edge in $A$ is in between them on one of the paths of the face and so on). \\

If you look at the resulting curves, you see a union of disjoint paths and cycles if you treat the midpoints of edges in $A$ as the vertex set of a new graph $G'$. This is because in this new graph all vertices (midpoints of edges in $A$) either have degree $2$ in the case that its corresponding edge in $G$ is contained in two finite faces, or degree $1$ in the case that its corresponding edge in $G$ is contained in one finite face. We call each path and cycle in $G'$ a \textbf{transversal} of $G$. \\

In the case that $G = T_n$, our curves can simply be straight line segments. Furthermore, in $T_n$ a transversal is unique since each finite face has three edges, so there can be at most one pair of edges in any finite face, and therefore it makes sense to talk about \textit{the} transversal of the symmetric difference between two cycles of the same signature in $T_n$. We give an example of a transversal of the symmetric difference of the two cycles in Figure 1 (located at the introduction). Note the symmetric difference between the two cycles in Figure 1 is $A(2)$. \\ 

\begin{center}
\begin{tikzpicture}[declare function={a=0.75;}]
\foreach \y in {0,...,5} 
\foreach \x in {0,...,\numexpr 5 - \y}
{
\draw ({(\x + (\y / 2))*a},{(sqrt(3) * \y / 2)*a}) coordinate (\x\y) ;
\filldraw[black] ({(\x + (\y / 2))*a},{(sqrt(3) * \y / 2)*a}) circle (2pt) ;
}

\draw[very thick] (10) -- (20) ;
\draw[very thick] (30) -- (40) ;
\draw[very thick] (10) -- (11) ;
\draw[very thick] (20) -- (21) ;
\draw[very thick] (30) -- (21) ;
\draw[very thick] (40) -- (31) ;
\draw[very thick] (11) -- (21) ;
\draw[very thick] (21) -- (31) ;
\draw[very thick] (01) -- (11) ;
\draw[very thick] (01) -- (02) ;
\draw[very thick] (02) -- (12) ;
\draw[very thick] (12) -- (11) ;
\draw[very thick] (31) -- (22) ;
\draw[very thick] (22) -- (32) ;
\draw[very thick] (32) -- (41) ;
\draw[very thick] (41) -- (31) ;
\draw[very thick] (12) -- (03) ;
\draw[very thick] (03) -- (04) ;
\draw[very thick] (04) -- (13) ;
\draw[very thick] (13) -- (12) ;
\draw[very thick] (22) -- (13) ;
\draw[very thick] (13) -- (14) ;
\draw[very thick] (14) -- (23) ;
\draw[very thick] (23) -- (22) ;

\draw ($(01)!0.5!(02)$) -- ($(01)!0.5!(11)$) ;
\draw ($(01)!0.5!(11)$) -- ($(10)!0.5!(11)$) ;
\draw ($(10)!0.5!(11)$) -- ($(10)!0.5!(20)$) ;
\draw ($(11)!0.5!(21)$) -- ($(20)!0.5!(21)$) ;
\draw ($(20)!0.5!(21)$) -- ($(30)!0.5!(21)$) ;
\draw ($(30)!0.5!(21)$) -- ($(21)!0.5!(31)$) ;
\draw ($(21)!0.5!(31)$) -- ($(31)!0.5!(22)$) ;
\draw ($(22)!0.5!(31)$) -- ($(22)!0.5!(32)$) ;
\draw ($(22)!0.5!(32)$) -- ($(22)!0.5!(23)$) ;
\draw ($(22)!0.5!(23)$) -- ($(22)!0.5!(13)$) ;
\draw ($(22)!0.5!(13)$) -- ($(12)!0.5!(03)$) ;
\draw ($(12)!0.5!(03)$) -- ($(12)!0.5!(02)$) ;
\draw ($(12)!0.5!(02)$) -- ($(11)!0.5!(12)$) ;
\draw ($(11)!0.5!(12)$) -- ($(11)!0.5!(21)$) ;
\draw ($(03)!0.5!(04)$) -- ($(13)!0.5!(04)$) ;
\draw ($(13)!0.5!(04)$) -- ($(13)!0.5!(14)$) ;
\draw ($(13)!0.5!(14)$) -- ($(23)!0.5!(14)$) ;
\draw ($(30)!0.5!(40)$) -- ($(40)!0.5!(31)$) ;
\draw ($(40)!0.5!(31)$) -- ($(31)!0.5!(41)$) ;
\draw ($(31)!0.5!(41)$) -- ($(41)!0.5!(32)$) ;
\end{tikzpicture}

\bigskip
Figure 6
\end{center}

\noindent
The edges of $A(2)$ are the thicker line segments, while the transversals are thinner. The following proposition was proven in \cite{bel}. \\

\begin{proposition}[Beluhov]
The size of the symmetric difference $A$ of of two cycles $C_1,C_2$ with the same signature in a plane graph is always divisible by $4$. 
\end{proposition}

\begin{proof}
As you see in Figure 6, the number of edges in $A$ any transversal intersects is divisible by $4$. The three path-transversals in Figure 6 intersect $4$ edges in $A$, and the cycle-transversal intersects with $12$ edges in $A$. In the general case, note that since a transversal alternates between edges in $C_1$ and in $C_2$ by definition, the number of edges $n_1$ in $C_1$ the transversal intersects is at most $1$ off from the number of edges $n_2$ in $C_2$ the transversal intersects, that is, $|n_1-n_2| \leq 1$. But also note that $n_1$ and $n_2$ must be even; a transversal cuts the plane graph into two parts and contains all the edges in between those two parts. Then since $C_1,C_2$ are cycles we have that ``what goes in must come out," and so $n_1,n_2$ must be even. This implies $n_1 = n_2$, so that the total number of edges that the transversal intersects is $n_1 + n_2 = 2n_1$ which is divisible by $4$, as $n_1$ is even. 
\end{proof}

\begin{corollary}
The greatest common divisor of the sizes of the symmetric difference between two cycles with the same signature in $T_n$ for any $n \in \mathbb{Z}^{+}$ is $12$. In other words, $12$ divides the size of the symmetric difference between any two cycles with the same signature in $T_n$ for any $n \in \mathbb{Z}^{+}$, and there is no number strictly larger than $12$ that also satisfies this.  
\end{corollary}

\begin{proof}
As mentioned in the Introduction, the symmetric difference of two cycles with the same signature is a totally even subset. In $T_n$, a totally even subset $A$ has a ``six-fold symmetry" as shown in Figure 5, and thus $6 \mid |A|$ and so $3 \mid |A|$ as well. If $A$ is the symmetric difference of two cycles with the same signature, then we know $4 \mid |A|$ by Proposition $2$. Thus $12 \mid |A|$. \\

To show $12$ is the \textit{greatest} common divisor, consider Figure 6, which is the symmetric difference of the two cycles with the same signature in $T_5$ in Figure 1. The symmetric difference has size $24$. We now give an example where the symmetric difference between two cycles of the same signature in $T_{11}$ is $60$. Note $\gcd(24,60) = 12$. \\

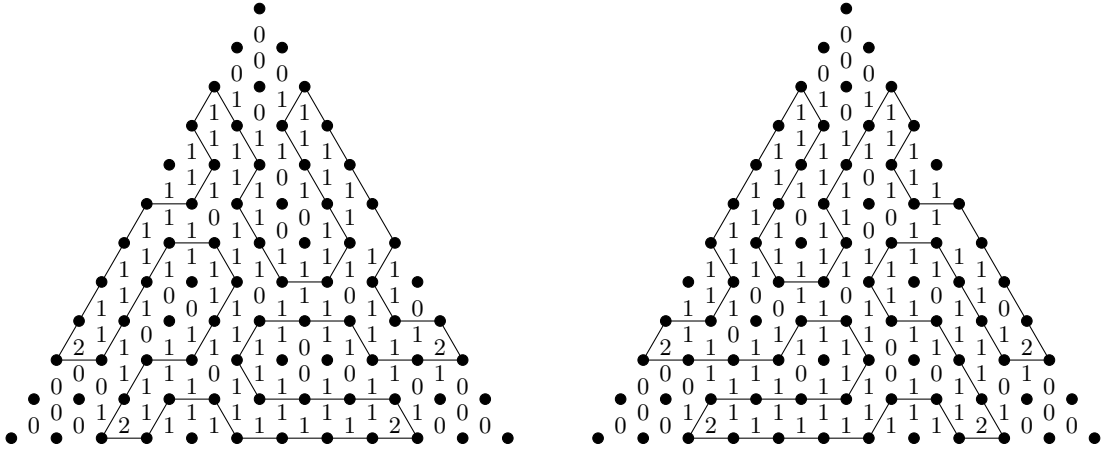
\begin{center}
\begin{tikzpicture}[declare function={a=0.6; f(\x,\y) = (\x + (\y / 2))*a ; g(\x,\y) =  (sqrt(3) * \y / 2)*a ; }]
\foreach \y in {0,...,11} 
\foreach \x in {0,...,\numexpr 11 - \y}
{
\draw ({f(\x,\y)},{g(\x,\y)}) coordinate (\x\y) ;
\filldraw[black] (\x\y) circle (2pt) ;
}

\draw (20) -- (30) ;
\draw (30) -- (31) ;
\draw (31) -- (41) ;
\draw (41) -- (50) ;
\draw (50) -- (90) ;
\draw (90) -- (81) ;
\draw (81) -- (51) ;
\draw (51) -- (42) ;
\draw (42) -- (43) ;
\draw (43) -- (63) ;
\draw (63) -- (72) ;
\draw (72) -- (92) ;
\draw (92) -- (83) ;
\draw (83) -- (73) ;
\draw (73) -- (64) ;
\draw (64) -- (65) ;
\draw (65) -- (29) ;
\draw (29) -- (28) ;
\draw (28) -- (55) ;
\draw (55) -- (54) ;
\draw (54) -- (44) ;
\draw (44) -- (26) ;
\draw (26) -- (27) ;
\draw (27) -- (09) ;
\draw (09) -- (08) ;
\draw (08) -- (17) ;
\draw (17) -- (16) ;
\draw (16) -- (06) ;
\draw (06) -- (02) ;
\draw (02) -- (12) ;
\draw (12) -- (15) -- (25) ;
\draw (25) -- (34) -- (32) -- (22) -- (20) ;
\draw (11*a,0) coordinate (b) ;

\draw (barycentric cs:00=1,01=1,10=1) node{\small 0} ;
\draw (barycentric cs:11=1,01=1,10=1) node{\small 0} ;
\draw (barycentric cs:11=1,20=1,10=1) node{\small 0} ;
\draw (barycentric cs:11=1,20=1,21=1) node{\small 1} ;
\draw (barycentric cs:30=1,20=1,21=1) node{\small 2} ;
\draw (barycentric cs:30=1,31=1,21=1) node{\small 1} ;
\draw (barycentric cs:30=1,31=1,40=1) node{\small 1} ;
\draw (barycentric cs:41=1,31=1,40=1) node{\small 1} ;
\draw (barycentric cs:41=1,50=1,40=1) node{\small 1} ;
\draw (barycentric cs:41=1,50=1,51=1) node{\small 1} ;
\draw (barycentric cs:60=1,50=1,51=1) node{\small 1} ;
\draw (barycentric cs:60=1,61=1,51=1) node{\small 1} ;
\draw (barycentric cs:60=1,61=1,70=1) node{\small 1} ;
\draw (barycentric cs:71=1,61=1,70=1) node{\small 1} ;
\draw (barycentric cs:71=1,80=1,70=1) node{\small 1} ;
\draw (barycentric cs:71=1,80=1,81=1) node{\small 1} ;
\draw (barycentric cs:90=1,80=1,81=1) node{\small 2} ;
\draw (barycentric cs:90=1,91=1,81=1) node{\small 1} ;
\draw (barycentric cs:90=1,91=1,100=1) node{\small 0} ;
\draw (barycentric cs:101=1,91=1,100=1) node{\small 0} ;
\draw (barycentric cs:101=1,b=1,100=1) node{\small 0} ;
\draw (barycentric cs:01=1,02=1,11=1) node{\small 0} ;
\draw (barycentric cs:12=1,02=1,11=1) node{\small 0} ;
\draw (barycentric cs:12=1,21=1,11=1) node{\small 0} ;
\draw (barycentric cs:12=1,21=1,22=1) node{\small 1} ;
\draw (barycentric cs:31=1,21=1,22=1) node{\small 1} ;
\draw (barycentric cs:31=1,32=1,22=1) node{\small 1} ;
\draw (barycentric cs:31=1,32=1,41=1) node{\small 1} ;
\draw (barycentric cs:42=1,32=1,41=1) node{\small 0} ;
\draw (barycentric cs:42=1,51=1,41=1) node{\small 1} ;
\draw (barycentric cs:42=1,51=1,52=1) node{\small 1} ;
\draw (barycentric cs:61=1,51=1,52=1) node{\small 1} ;
\draw (barycentric cs:61=1,62=1,52=1) node{\small 0} ;
\draw (barycentric cs:61=1,62=1,71=1) node{\small 1} ;
\draw (barycentric cs:72=1,62=1,71=1) node{\small 0} ;
\draw (barycentric cs:72=1,81=1,71=1) node{\small 1} ;
\draw (barycentric cs:72=1,81=1,82=1) node{\small 1} ;
\draw (barycentric cs:91=1,81=1,82=1) node{\small 0} ;
\draw (barycentric cs:91=1,92=1,82=1) node{\small 1} ;
\draw (barycentric cs:91=1,92=1,101=1) node{\small 0} ;
\draw (barycentric cs:02=1,03=1,12=1) node{\small 2} ;
\draw (barycentric cs:13=1,03=1,12=1) node{\small 1} ;
\draw (barycentric cs:13=1,22=1,12=1) node{\small 1} ;
\draw (barycentric cs:13=1,22=1,23=1) node{\small 0} ;
\draw (barycentric cs:32=1,22=1,23=1) node{\small 1} ;
\draw (barycentric cs:32=1,33=1,23=1) node{\small 1} ;
\draw (barycentric cs:32=1,33=1,42=1) node{\small 1} ;
\draw (barycentric cs:43=1,33=1,42=1) node{\small 1} ;
\draw (barycentric cs:43=1,52=1,42=1) node{\small 1} ;
\draw (barycentric cs:43=1,52=1,53=1) node{\small 1} ;
\draw (barycentric cs:62=1,52=1,53=1) node{\small 0} ;
\draw (barycentric cs:62=1,63=1,53=1) node{\small 1} ;
\draw (barycentric cs:62=1,63=1,72=1) node{\small 1} ;
\draw (barycentric cs:73=1,63=1,72=1) node{\small 1} ;
\draw (barycentric cs:73=1,82=1,72=1) node{\small 1} ;
\draw (barycentric cs:73=1,82=1,83=1) node{\small 1} ;
\draw (barycentric cs:92=1,82=1,83=1) node{\small 2} ;
\draw (barycentric cs:03=1,04=1,13=1) node{\small 1} ;
\draw (barycentric cs:14=1,04=1,13=1) node{\small 1} ;
\draw (barycentric cs:14=1,23=1,13=1) node{\small 1} ;
\draw (barycentric cs:14=1,23=1,24=1) node{\small 0} ;
\draw (barycentric cs:33=1,23=1,24=1) node{\small 0} ;
\draw (barycentric cs:33=1,34=1,24=1) node{\small 1} ;
\draw (barycentric cs:33=1,34=1,43=1) node{\small 1} ;
\draw (barycentric cs:44=1,34=1,43=1) node{\small 0} ;
\draw (barycentric cs:44=1,53=1,43=1) node{\small 1} ;
\draw (barycentric cs:44=1,53=1,54=1) node{\small 1} ;
\draw (barycentric cs:63=1,53=1,54=1) node{\small 1} ;
\draw (barycentric cs:63=1,64=1,54=1) node{\small 0} ;
\draw (barycentric cs:63=1,64=1,73=1) node{\small 1} ;
\draw (barycentric cs:74=1,64=1,73=1) node{\small 1} ;
\draw (barycentric cs:74=1,83=1,73=1) node{\small 0} ;
\draw (barycentric cs:04=1,05=1,14=1) node{\small 1} ;
\draw (barycentric cs:15=1,05=1,14=1) node{\small 1} ;
\draw (barycentric cs:15=1,24=1,14=1) node{\small 1} ;
\draw (barycentric cs:15=1,24=1,25=1) node{\small 1} ;
\draw (barycentric cs:34=1,24=1,25=1) node{\small 1} ;
\draw (barycentric cs:34=1,35=1,25=1) node{\small 1} ;
\draw (barycentric cs:34=1,35=1,44=1) node{\small 1} ;
\draw (barycentric cs:45=1,35=1,44=1) node{\small 1} ;
\draw (barycentric cs:45=1,54=1,44=1) node{\small 1} ;
\draw (barycentric cs:45=1,54=1,55=1) node{\small 1} ;
\draw (barycentric cs:64=1,54=1,55=1) node{\small 1} ;
\draw (barycentric cs:64=1,65=1,55=1) node{\small 1} ;
\draw (barycentric cs:64=1,65=1,74=1) node{\small 1} ;
\draw (barycentric cs:05=1,06=1,15=1) node{\small 1} ;
\draw (barycentric cs:16=1,06=1,15=1) node{\small 1} ;
\draw (barycentric cs:16=1,25=1,15=1) node{\small 1} ;
\draw (barycentric cs:16=1,25=1,26=1) node{\small 0} ;
\draw (barycentric cs:35=1,25=1,26=1) node{\small 1} ;
\draw (barycentric cs:35=1,36=1,26=1) node{\small 1} ;
\draw (barycentric cs:35=1,36=1,45=1) node{\small 0} ;
\draw (barycentric cs:46=1,36=1,45=1) node{\small 0} ;
\draw (barycentric cs:46=1,55=1,45=1) node{\small 1} ;
\draw (barycentric cs:46=1,55=1,56=1) node{\small 1} ;
\draw (barycentric cs:06=1,07=1,16=1) node{\small 1} ;
\draw (barycentric cs:17=1,07=1,16=1) node{\small 1} ;
\draw (barycentric cs:17=1,26=1,16=1) node{\small 1} ;
\draw (barycentric cs:17=1,26=1,27=1) node{\small 1} ;
\draw (barycentric cs:36=1,26=1,27=1) node{\small 1} ;
\draw (barycentric cs:36=1,37=1,27=1) node{\small 0} ;
\draw (barycentric cs:36=1,37=1,46=1) node{\small 1} ;
\draw (barycentric cs:47=1,37=1,46=1) node{\small 1} ;
\draw (barycentric cs:47=1,56=1,46=1) node{\small 1} ;
\draw (barycentric cs:07=1,08=1,17=1) node{\small 1} ;
\draw (barycentric cs:18=1,08=1,17=1) node{\small 1} ;
\draw (barycentric cs:18=1,27=1,17=1) node{\small 1} ;
\draw (barycentric cs:18=1,27=1,28=1) node{\small 1} ;
\draw (barycentric cs:37=1,27=1,28=1) node{\small 1} ;
\draw (barycentric cs:37=1,38=1,28=1) node{\small 1} ;
\draw (barycentric cs:37=1,38=1,47=1) node{\small 1} ;
\draw (barycentric cs:08=1,09=1,18=1) node{\small 1} ;
\draw (barycentric cs:19=1,09=1,18=1) node{\small 1} ;
\draw (barycentric cs:19=1,28=1,18=1) node{\small 0} ;
\draw (barycentric cs:19=1,28=1,29=1) node{\small 1} ;
\draw (barycentric cs:38=1,28=1,29=1) node{\small 1} ;
\draw (barycentric cs:09=1,010=1,19=1) node{\small 0} ;
\draw (barycentric cs:110=1,010=1,19=1) node{\small 0} ;
\draw (barycentric cs:110=1,29=1,19=1) node{\small 0} ;
\draw (barycentric cs:010=1,011=1,110=1) node{\small 0} ;

\foreach \y in {0,...,11} 
\foreach \x in {0,...,\numexpr 11 - \y}
{
\draw ({f(\x,\y)+13*a},{g(\x,\y)}) coordinate (\x\y) ;
\filldraw[black] (\x\y) circle (2pt) ;
}

\draw (20) -- (60) ;
\draw (60) -- (61) ;
\draw (61) -- (71) ;
\draw (71) -- (80) ;
\draw (80) -- (90) ;
\draw (90) -- (63) ;
\draw (63) -- (53) ;
\draw (53) -- (44) ;
\draw (44) -- (45) ;
\draw (45) -- (55) ;
\draw (55) -- (82) ;
\draw (82) -- (92) ;
\draw (92) -- (56) ;
\draw (56) -- (46) -- (37) -- (38) -- (29) -- (26) -- (35) -- (34) -- (24) -- (15) -- (18) -- (09) ;
\draw (09) -- (05) -- (14) -- (13) -- (03) -- (02) -- (32) -- (33) -- (43) -- (52) -- (51) -- (21) -- (20) ;

\draw (11*a + 13*a,0) coordinate (b) ;

\draw (barycentric cs:00=1,01=1,10=1) node{\small 0} ;
\draw (barycentric cs:11=1,01=1,10=1) node{\small 0} ;
\draw (barycentric cs:11=1,20=1,10=1) node{\small 0} ;
\draw (barycentric cs:11=1,20=1,21=1) node{\small 1} ;
\draw (barycentric cs:30=1,20=1,21=1) node{\small 2} ;
\draw (barycentric cs:30=1,31=1,21=1) node{\small 1} ;
\draw (barycentric cs:30=1,31=1,40=1) node{\small 1} ;
\draw (barycentric cs:41=1,31=1,40=1) node{\small 1} ;
\draw (barycentric cs:41=1,50=1,40=1) node{\small 1} ;
\draw (barycentric cs:41=1,50=1,51=1) node{\small 1} ;
\draw (barycentric cs:60=1,50=1,51=1) node{\small 1} ;
\draw (barycentric cs:60=1,61=1,51=1) node{\small 1} ;
\draw (barycentric cs:60=1,61=1,70=1) node{\small 1} ;
\draw (barycentric cs:71=1,61=1,70=1) node{\small 1} ;
\draw (barycentric cs:71=1,80=1,70=1) node{\small 1} ;
\draw (barycentric cs:71=1,80=1,81=1) node{\small 1} ;
\draw (barycentric cs:90=1,80=1,81=1) node{\small 2} ;
\draw (barycentric cs:90=1,91=1,81=1) node{\small 1} ;
\draw (barycentric cs:90=1,91=1,100=1) node{\small 0} ;
\draw (barycentric cs:101=1,91=1,100=1) node{\small 0} ;
\draw (barycentric cs:101=1,b=1,100=1) node{\small 0} ;
\draw (barycentric cs:01=1,02=1,11=1) node{\small 0} ;
\draw (barycentric cs:12=1,02=1,11=1) node{\small 0} ;
\draw (barycentric cs:12=1,21=1,11=1) node{\small 0} ;
\draw (barycentric cs:12=1,21=1,22=1) node{\small 1} ;
\draw (barycentric cs:31=1,21=1,22=1) node{\small 1} ;
\draw (barycentric cs:31=1,32=1,22=1) node{\small 1} ;
\draw (barycentric cs:31=1,32=1,41=1) node{\small 1} ;
\draw (barycentric cs:42=1,32=1,41=1) node{\small 0} ;
\draw (barycentric cs:42=1,51=1,41=1) node{\small 1} ;
\draw (barycentric cs:42=1,51=1,52=1) node{\small 1} ;
\draw (barycentric cs:61=1,51=1,52=1) node{\small 1} ;
\draw (barycentric cs:61=1,62=1,52=1) node{\small 0} ;
\draw (barycentric cs:61=1,62=1,71=1) node{\small 1} ;
\draw (barycentric cs:72=1,62=1,71=1) node{\small 0} ;
\draw (barycentric cs:72=1,81=1,71=1) node{\small 1} ;
\draw (barycentric cs:72=1,81=1,82=1) node{\small 1} ;
\draw (barycentric cs:91=1,81=1,82=1) node{\small 0} ;
\draw (barycentric cs:91=1,92=1,82=1) node{\small 1} ;
\draw (barycentric cs:91=1,92=1,101=1) node{\small 0} ;
\draw (barycentric cs:02=1,03=1,12=1) node{\small 2} ;
\draw (barycentric cs:13=1,03=1,12=1) node{\small 1} ;
\draw (barycentric cs:13=1,22=1,12=1) node{\small 1} ;
\draw (barycentric cs:13=1,22=1,23=1) node{\small 0} ;
\draw (barycentric cs:32=1,22=1,23=1) node{\small 1} ;
\draw (barycentric cs:32=1,33=1,23=1) node{\small 1} ;
\draw (barycentric cs:32=1,33=1,42=1) node{\small 1} ;
\draw (barycentric cs:43=1,33=1,42=1) node{\small 1} ;
\draw (barycentric cs:43=1,52=1,42=1) node{\small 1} ;
\draw (barycentric cs:43=1,52=1,53=1) node{\small 1} ;
\draw (barycentric cs:62=1,52=1,53=1) node{\small 0} ;
\draw (barycentric cs:62=1,63=1,53=1) node{\small 1} ;
\draw (barycentric cs:62=1,63=1,72=1) node{\small 1} ;
\draw (barycentric cs:73=1,63=1,72=1) node{\small 1} ;
\draw (barycentric cs:73=1,82=1,72=1) node{\small 1} ;
\draw (barycentric cs:73=1,82=1,83=1) node{\small 1} ;
\draw (barycentric cs:92=1,82=1,83=1) node{\small 2} ;
\draw (barycentric cs:03=1,04=1,13=1) node{\small 1} ;
\draw (barycentric cs:14=1,04=1,13=1) node{\small 1} ;
\draw (barycentric cs:14=1,23=1,13=1) node{\small 1} ;
\draw (barycentric cs:14=1,23=1,24=1) node{\small 0} ;
\draw (barycentric cs:33=1,23=1,24=1) node{\small 0} ;
\draw (barycentric cs:33=1,34=1,24=1) node{\small 1} ;
\draw (barycentric cs:33=1,34=1,43=1) node{\small 1} ;
\draw (barycentric cs:44=1,34=1,43=1) node{\small 0} ;
\draw (barycentric cs:44=1,53=1,43=1) node{\small 1} ;
\draw (barycentric cs:44=1,53=1,54=1) node{\small 1} ;
\draw (barycentric cs:63=1,53=1,54=1) node{\small 1} ;
\draw (barycentric cs:63=1,64=1,54=1) node{\small 0} ;
\draw (barycentric cs:63=1,64=1,73=1) node{\small 1} ;
\draw (barycentric cs:74=1,64=1,73=1) node{\small 1} ;
\draw (barycentric cs:74=1,83=1,73=1) node{\small 0} ;
\draw (barycentric cs:04=1,05=1,14=1) node{\small 1} ;
\draw (barycentric cs:15=1,05=1,14=1) node{\small 1} ;
\draw (barycentric cs:15=1,24=1,14=1) node{\small 1} ;
\draw (barycentric cs:15=1,24=1,25=1) node{\small 1} ;
\draw (barycentric cs:34=1,24=1,25=1) node{\small 1} ;
\draw (barycentric cs:34=1,35=1,25=1) node{\small 1} ;
\draw (barycentric cs:34=1,35=1,44=1) node{\small 1} ;
\draw (barycentric cs:45=1,35=1,44=1) node{\small 1} ;
\draw (barycentric cs:45=1,54=1,44=1) node{\small 1} ;
\draw (barycentric cs:45=1,54=1,55=1) node{\small 1} ;
\draw (barycentric cs:64=1,54=1,55=1) node{\small 1} ;
\draw (barycentric cs:64=1,65=1,55=1) node{\small 1} ;
\draw (barycentric cs:64=1,65=1,74=1) node{\small 1} ;
\draw (barycentric cs:05=1,06=1,15=1) node{\small 1} ;
\draw (barycentric cs:16=1,06=1,15=1) node{\small 1} ;
\draw (barycentric cs:16=1,25=1,15=1) node{\small 1} ;
\draw (barycentric cs:16=1,25=1,26=1) node{\small 0} ;
\draw (barycentric cs:35=1,25=1,26=1) node{\small 1} ;
\draw (barycentric cs:35=1,36=1,26=1) node{\small 1} ;
\draw (barycentric cs:35=1,36=1,45=1) node{\small 0} ;
\draw (barycentric cs:46=1,36=1,45=1) node{\small 0} ;
\draw (barycentric cs:46=1,55=1,45=1) node{\small 1} ;
\draw (barycentric cs:46=1,55=1,56=1) node{\small 1} ;
\draw (barycentric cs:06=1,07=1,16=1) node{\small 1} ;
\draw (barycentric cs:17=1,07=1,16=1) node{\small 1} ;
\draw (barycentric cs:17=1,26=1,16=1) node{\small 1} ;
\draw (barycentric cs:17=1,26=1,27=1) node{\small 1} ;
\draw (barycentric cs:36=1,26=1,27=1) node{\small 1} ;
\draw (barycentric cs:36=1,37=1,27=1) node{\small 0} ;
\draw (barycentric cs:36=1,37=1,46=1) node{\small 1} ;
\draw (barycentric cs:47=1,37=1,46=1) node{\small 1} ;
\draw (barycentric cs:47=1,56=1,46=1) node{\small 1} ;
\draw (barycentric cs:07=1,08=1,17=1) node{\small 1} ;
\draw (barycentric cs:18=1,08=1,17=1) node{\small 1} ;
\draw (barycentric cs:18=1,27=1,17=1) node{\small 1} ;
\draw (barycentric cs:18=1,27=1,28=1) node{\small 1} ;
\draw (barycentric cs:37=1,27=1,28=1) node{\small 1} ;
\draw (barycentric cs:37=1,38=1,28=1) node{\small 1} ;
\draw (barycentric cs:37=1,38=1,47=1) node{\small 1} ;
\draw (barycentric cs:08=1,09=1,18=1) node{\small 1} ;
\draw (barycentric cs:19=1,09=1,18=1) node{\small 1} ;
\draw (barycentric cs:19=1,28=1,18=1) node{\small 0} ;
\draw (barycentric cs:19=1,28=1,29=1) node{\small 1} ;
\draw (barycentric cs:38=1,28=1,29=1) node{\small 1} ;
\draw (barycentric cs:09=1,010=1,19=1) node{\small 0} ;
\draw (barycentric cs:110=1,010=1,19=1) node{\small 0} ;
\draw (barycentric cs:110=1,29=1,19=1) node{\small 0} ;
\draw (barycentric cs:010=1,011=1,110=1) node{\small 0} ;

\end{tikzpicture} 
\end{center}

The symmetric difference of the two cycles is $A(4) \triangle A(5)$, shown below. \\

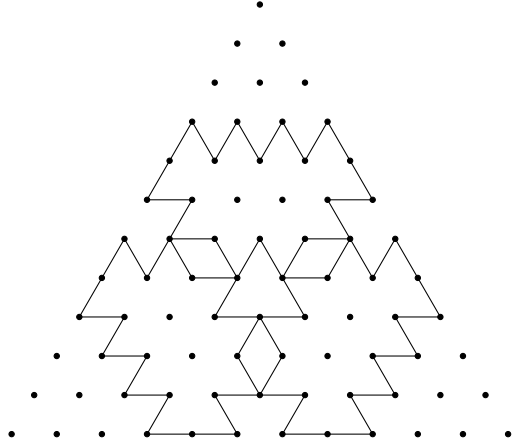
\begin{center}
\begin{tikzpicture}[declare function={a=0.6; f(\x,\y) = (\x + (\y / 2))*a ; g(\x,\y) =  (sqrt(3) * \y / 2)*a ; }]

\foreach \y in {0,...,11} 
\foreach \x in {0,...,\numexpr 11 - \y}
{
\draw ({f(\x,\y)},{g(\x,\y)}) coordinate (\x\y) ;
\filldraw[black] ({f(\x,\y)},{g(\x,\y)}) circle (1pt) ;
}

\draw (30) -- (50) -- (41) -- (61) -- (60) -- (80) -- (71) -- (81) -- (72) ;
\draw (72) -- (82) -- (73) -- (83) -- (65) ;
\draw (65) -- (64) -- (46) -- (56) ;
\draw (56) -- (38) -- (37) -- (28) -- (27) -- (18) -- (17) -- (08) -- (06) ;
\draw (06) -- (16) -- (14) -- (05) -- (03) -- (13) ;
\draw (13) -- (12) -- (22) -- (21) -- (31) -- (30) ;
\draw (51) -- (42) -- (43) -- (52) -- (51) ;
\draw (33) -- (53) -- (35) -- (33) ;
\draw (44) -- (54) -- (55) -- (45) -- (44) ;
\draw (15) -- (25) -- (34) -- (24) -- (15) ;
\end{tikzpicture}
\end{center}

\noindent
By Theorem 4, $A(4) \triangle A(5)$ has $12 \cdot (1 + 4) \cdot (1) = 60$ edges. 
\end{proof}

\begin{proposition}
Let $A = A(i_1) \triangle A(i_2) \triangle \cdots \triangle A(i_k)$ in $T_n$. If $i_1$ is odd, then $A$ is not the symmetric difference of two cycles with the same signature.
\end{proposition}

\begin{proof}
Consider the zig-zag of edges $\{(i_1,1),(i_1+1,1)\}, \{(i_1+1,1),(i_1,2)\}, \{(i_1,2),(i_1+1,2)\},\{(i_1+1,2),(i_1,3)\}$ and so on, that is, the set of edges $E_1 \cup E_2$ where

$$E_1 = \{\{(x,y),(x+1,y)\}\} : x + y = i_1 + 1\}$$

$$E_2 = \{\{(x,y),(x,y+1)\} : x + y = i_1 + 1\}.$$

Note $E_1 \cup E_2 \subseteq A \cap A(i_1)$, because for $s > 1$ we have $A(i_s) \cap (E_1 \cup E_2) = \{\}$ since $i_s > i_1$, which can be seen by the construction with the three regions $R_1,R_2,R_3$ given in Theorem 3. If $A$ is the symmetric difference between two cycles of the same signature, then there is a transversal that intersects all and only the edges in $E_1 \cup E_2$ ($E_1 \cup E_2$ forms a zig-zag path). Note $|E_1 \cup E_2| = 2i_1$, and since $i_1$ is odd we have $4 \nmid 2i_1$. Thus the number of edges the transversal intersects is not divisible by $4$, contradicting the proof of Proposition 2.
\end{proof}

\begin{center}
\section{Concluding Remarks}
\end{center}

There are some key statements proved by Beluhov \cite{bel} about rectangular grids with square cells that we couldn't find and prove an analogue of for triangular grids. For example, we wanted to prove that given a pair of distinct cycles with the same signature in $T_n$, we can construct a pair of distinct cycles with the same signature in $T_{mn}$ where $m \in \mathbb{Z}^{+}$. However there were obstacles that prevented us from doing so, and now we doubt whether the claim is even true for triangular grids (something similar is true for rectangular grids with square cells, which Beluhov proved). But we do think there should be a way to construct larger pairs of distinct cycles with the same signature from smaller pairs in triangular grids, but we just haven't found it. Also, we were able to prove the following analogue of a key lemma Beluhov used to prove the statement above (for rectangular grids). \\

\begin{remark}
If $C_1$ and $C_2$ have the same signature in $T_n$, then if a finite face has two edges of $C_1 \triangle C_2$, then one of these two edges must belong to $C_1 \setminus C_2$ and the other must belong to $C_2 \setminus C_1$.
\end{remark}

\begin{lemma}
Let $C_1$ and $C_2$ be two distinct cycles with the same signature in $T_n$. Then there exists two distinct cycles $C_1'$ and $C_2'$ with the same signature in $T_n$, such that $C_1'$ and $C_2'$ share at least one edge on each of the three sides of $T_n$. 
\end{lemma}

\begin{proof}
If $C_1$ and $C_2$ already share edges on each side of $T_n$, we are done. If both $C_1$ and $C_2$ contain no edges on a certain side, and without loss of generality let it be the bottom side. Then $C_1 = C_2$, as $C_1 \triangle C_2$ is a totally even subset with no edges on the bottom side, and hence $C_1 \triangle C_2 = \{\}$ by Remark 1. Now without loss of generality, assume $C_1$ and $C_2$ do not share edges on the bottom side, and at least one of $C_1$ or $C_2$ contains an edge of the bottom side. We construct distinct $C_1'$ and $C_2'$ with the same signature that shares an edge on the bottom side. There is analogous construction for the left and right sides. Let $\{(i,1),(i+1,1)\}$ be the left-most (minimal $i$) edge of $C_1 \cup C_2$ on the bottom side, and without loss of generality assume it belongs to $C_1$. Let $A = C_1 \triangle C_2$, so $A$ is totally even. Since $\{(i,1),(i+1,1)\} \in A$ is the leftmost edge of the bottom side of $A$, we have $A = A(i) \triangle \cdots$ by Remark 1. \\

Recall the zig-zag path in Proposition 3 with edge set $E_1 \cup E_2$ where $i_1 = i$. Then $E_1 \cup E_2 \subseteq A$, which can be seen by the construction with the three regions $R_1,R_2,R_3$ given in Theorem 3. From the zig-zag path $E_1 \cup E_2$ and Remark 2, we know that $\{(i+1,1),(i,2)\} \in C_2 \setminus C_1$. If $i = 1$ we get a contradiction, as the vertex $(1,1)$ would be incident to exactly one edge of $C_2$, so we may assume $i > 1$. Note that since $\{(i,1),(i+1,1)\}$ is the left-most edge on the bottom side of $A = C_1 \triangle C_2$, we have $A$ does not contain any edges in the equilateral triangular region with vertices $(1,1),(i,1),(1,i)$, and this can be seen by the construction with the three regions $R_1,R_2,R_3$ given in Theorem 3. Note that we must have $\{(i,1),(i-1,2)\} \in C_1 \cap C_2$ because $C_1$ and $C_2$ are cycles and they cannot contain the edge $\{(i-1,1),(i,1)\}$, because we assumed $\{(i,1),(i+1,1)\}$ was the left-most edge. Consider the minimum $j$ such that $\{(j,1), (j-1,2)\} \in C_1 \cap C_2$, which exists because $\{(i,1),(i-1,2)\} \in C_1 \cap C_2$. We can delete the edge $\{(j,1), (j-1,2)\}$ from both $C_1$ and $C_2$ and replace them with the two edges $\{(j-1,1),(j,1)\}$ and $\{(j-1,1),(j-1,2)\}$ in both $C_1$ and $C_2$ if they're not already in there. This gives us two cycles $C_1',C_2'$ that have the same signature and share the edge $\{(j-1,1),(j,1)\}$ on the bottom side. 
\end{proof}

Ultimately we wanted to completely determine for which $n$ does $T_n$ have a pair of distinct cycles with the same signature in $T_n$, which Beluhov was able to do for rectangular grids, and which we were not able to achieve.

\section{Conflict of Interest}

On behalf of all authors, the corresponding author states that there is no conflict of interest. \\

\section{Data Availability Statement}

No data are associated with this article. \\

\begin{center}
\section{Acknowledgments}
\end{center}

This research was conducted at the University of Minnesota Duluth REU and was supported by Jane Street Capital, NSF-DMS (Grant 1949884), Ray Sidney, and Eric Wepsic. We thank Joe Gallian, Colin Defant, Noah Kravitz, Maya Sankar, Mitchell Lee, and Katalin Berlow for feedback.  

\bibliographystyle{plain}
\bibliography{refs} 

\noindent
Department of Mathematics, Carnegie Mellon University, Pittsburgh, PA 15213, USA 
\end{document}